\documentclass[a4paper,10pt]{article}

\usepackage{authblk}
\usepackage{color, graphics}

\usepackage{graphicx}

\usepackage[T1]{fontenc}
\usepackage[latin1]{inputenc}
\usepackage[english]{babel}
\usepackage{geometry}
\usepackage{indentfirst} 
\title{Periodic solutions to the  Cahn-Hilliard  \\ equation in the plane}
\date{\today}
\author{Andrea Malchiodi\thanks{Scuola Normale Superiore, Piazza dei Cavalieri 7, 56126 Pisa. e-mail: andrea.malchiodi@sns.it}, Rainer Mandel\thanks{KIT, Institut für Analysis, Englerstraße 2, 76131 Karlsruhe, Germany. e-mail: Rainer.Mandel@kit.edu}, Matteo Rizzi\thanks{Instituto de Matemáticas, UNAM,
Área de la Investigación Científica, Circuito exterior, Ciudad Universitaria, 04510, México. e-mail: mrizzi@im.unam.mx}, }

\usepackage{amsmath,amsfonts,amssymb,amsthm}
\usepackage{cases}
\usepackage{latexsym}

\newcommand{\N}{\mathbb{N}}

\newcommand{\R}{\mathbb{R}}

\newcommand{\tr}{\text{tr}}

\newcommand{\eps}{\varepsilon}

\newtheorem{theorem}{Theorem}
\newtheorem{proposition}[theorem]{Proposition}
\newtheorem{lemma}[theorem]{Lemma}

\newtheorem{remark}[theorem]{Remark}

\newenvironment{pfn}{\noindent{\em Proof}}{\rule{2mm}{2mm}\medskip}

\begin{document}

\maketitle

\begin{abstract}
In this paper we construct entire solutions to the Cahn-Hilliard equation
$-\Delta(-\Delta u+W^{'}(u))+W^{''}(u)(-\Delta u+W^{'}(u))=0$ 
in the Euclidean plane, where $W(u)$ is the standard double-well potential $\frac{1}{4} (1-u^2)^2$. 
Such solutions have a non-trivial profile that shadows a Willmore planar curve, and converge 
uniformly to $\pm 1$ as $x_2 \to \pm \infty$. These solutions give a counterexample to 
the counterpart of Gibbons'  conjecture for the fourth-order counterpart of the Allen-Cahn equation.  
We also study the $x_2$-derivative of these 
solutions using the special structure of  Willmore's equation. 
\end{abstract}

\tableofcontents

\section{Introduction} 

The Cahn-Hilliard equation was introduced in \cite{CH} to model phase separation of binary fluids. 
Typically, in experiments, a mixture of fluids tends to gradually self-arrange into more regular oscillatory patterns, 
with a sharp transition from one component to the other. Applications of this model include complex 
fluids and soft matter, such as polymer science. The goal of this paper is to rigorously 
construct planar solutions  modelling wiggly transient  patterns exhibited by the equation, and 
to relate them to some existing literature concerning the {\em Allen-Cahn equation}, 
a second-order counterpart of the Cahn-Hilliard describing phase separation in alloys. 

Let us begin by recalling some basic features about the Allen-Cahn equation 
\begin{equation}
-\Delta u=u-u^{3},
\label{allen-cahn} 
\end{equation}
introduced in \cite{AC}.   Here $u$ represents, up to an affine transformation, the density of one of the components of an alloy, 
whose energy per unit volume is given by a double-well potential $W$ 
\begin{equation}\label{eq:W}
 W(u) = \frac{1}{4} (1-u^2)^2. 
\end{equation}
Global minimizers (for example taken among functions with a prescribed average) of the integral of $W$ 
consist of the functions attaining only the values $\pm 1$. Since of course this set of functions has no 
structure whatsoever, usually  a regularization of the energy of the following type  is considered 
\begin{eqnarray}\notag
E_{\varepsilon}(u)=\int_{\Omega}\bigg(\frac{\varepsilon}{2}|\nabla u|^{2}+\frac{(1-u^{2})^{2}}{4\varepsilon}\bigg)dx, 
\end{eqnarray}
which penalizes too frequent phase transitions.

It was shown in \cite{MM} that under suitable assumptions  $E_{\varepsilon}$ Gamma-converges as $\eps \to 0$ to the 
perimeter functional and therefore its critical points are expected to have transitions approximating 
surfaces with zero mean curvature. In particular, minimizers for $E_\varepsilon$ 
should produce interfaces that are stable minimal surfaces, see \cite{Mo}, \cite{St} (and also \cite{HT}). The relation between stability of 
solutions to \eqref{allen-cahn} and their monotonicity has been the subject of several investigations, 
see for example \cite{AAC}, \cite{GG}, \cite{Sa}. In particular  a celebrated conjecture 
by E. De Giorgi (\cite{DG}) states that  solutions to \eqref{allen-cahn} that are monotone in some 
direction should depend on one variable only in dimension $n \leq 8$. This restriction on $n$ 
is crucial, since in large dimension there exist stable minimal surfaces that are not planar, see 
\cite{BDG}, and recently some entire solutions modelled on them were constructed in 
\cite{dPKW}.  Further solutions with non-trivial profiles were produced for example in \cite{ADW}, \cite{CT}, \cite{DFP}, 
\cite{DKPW}, \cite{DMP}.

Another related conjecture named as the {\em Gibbons conjecture}, motivated by problems in 
cosmology, asserts that solutions 
to \eqref{allen-cahn} such that 
 $$
   u(x',x_n) \to \pm 1 \quad \hbox{ as } x_n \to \pm \infty \qquad \hbox{ uniformly for } x' \in \R^{n-1}, 
 $$
should also be one-dimensional. This conjecture was indeed fully proved in all dimensions, see 
\cite{BBG}, \cite{BHM}, \cite{Fa}, \cite{GG}.

\

We turn next to the Cahn-Hilliard equation
\begin{equation}
-\Delta(-\Delta u+W^{'}(u))+W^{''}(u)(-\Delta u+W^{'}(u))=0. 
\label{Cahn-Hilliard}
\end{equation}
%
%
Similarly to \eqref{allen-cahn}, also this equation is variational: introducing a scaling 
parameter $\eps > 0$,  its Euler-Lagrange functional is 
given by 
$$
 \mathcal{W}_{\varepsilon}(u)=  \frac{1}{2\varepsilon}\int_{\Omega}\big(\varepsilon\Delta u-\frac{W^{'}(u)}{\varepsilon}\big)^{2}dx. 
$$
Notice that when the integrand vanishes identically $u$ solves a scaled version of 
\eqref{allen-cahn}. As for $E_\eps$, also $\mathcal{W}_{\varepsilon}$ has a geometric 
interpretation as $\eps \to 0$. 
Although the characterization of Gamma-limit is not as complete as for the Allen-Cahn 
equation, some partial results are known about convergence to (a multiple of) the {\em Willmore energy} 
of the limit interface, i.e. the integral of the mean curvature squared 
\begin{eqnarray}\notag
\mathcal{W}_0(u)=
\int_{\partial E\cap\Omega}H_{\partial E}^{2}(y)d\mathcal{H}^{N-1}. 
\end{eqnarray}
In \cite{BP} G. Bellettini and M. Paolini proved the $\Gamma-\limsup$ inequality for smooth Willmore hypersurfaces, while the $\Gamma-\liminf$ inequality has been proved in dimension $N=2,3$ by M. Röger and R. Schätzle in \cite{RS}, and, independently, in dimension $N=2$, by Y. Nagase and Y. Tonegawa in \cite{NT}. It is an open problem to 
study in higher dimension, as well as to understand for which class of sets the Gamma-limit might exist. 
\\

Apart from the relation to the Cahn-Hilliard energy, the Willmore functional appears as 
bending energy of plates and membranes in mechanics and in biology, and it also enters in general 
relativity as the {\em Hawking mass} of a portion of space-time. This energy has also interest 
in geometry, since it is  invariant under M\"obius transformations.  Critical surfaces of 
$\mathcal{W}$ are called \textit{Willmore hypersurfaces}, and they are known to 
exist for any genus, see \cite{BK}. The Euler equation  is
\begin{eqnarray}\notag
-\Delta_{\Sigma}H=\frac{1}{2}H^{3}-2HK. 
\label{willmoreeq}
\end{eqnarray}
Interesting Willmore surfaces are {\em Clifford tori} (and their M\"obius transformations), that can be obtained 
by rotating around the $z$-axis a circle of radius $1$ and with center at distance $\sqrt{2}$ from the axis. 
Due to a recent result in \cite{MN}, establishing the so-called {\em Willmore conjecture}, this torus minimizes the 
Willmore energy among all surfaces of positive genus. In \cite{Ri}, up to a small Lagrange multiplier, solutions of 
\eqref{Cahn-Hilliard} in $\R^3$ were found with interfaces approaching a Clifford torus,  converging to $-1$ 
in its interior and to $+1$ on its exterior.  

Here we will show existence of solutions to \eqref{Cahn-Hilliard} in the plane with an  
interface periodic in $x_1$, shadowing a $T$-periodic (in the arc-length parameter)  Willmore curve $\gamma_T$, whose profile is given in the picture below. 
\begin{figure}[h]\label{fig:graph} 
\begin{center}
 \includegraphics[angle=0,width=12.0cm]{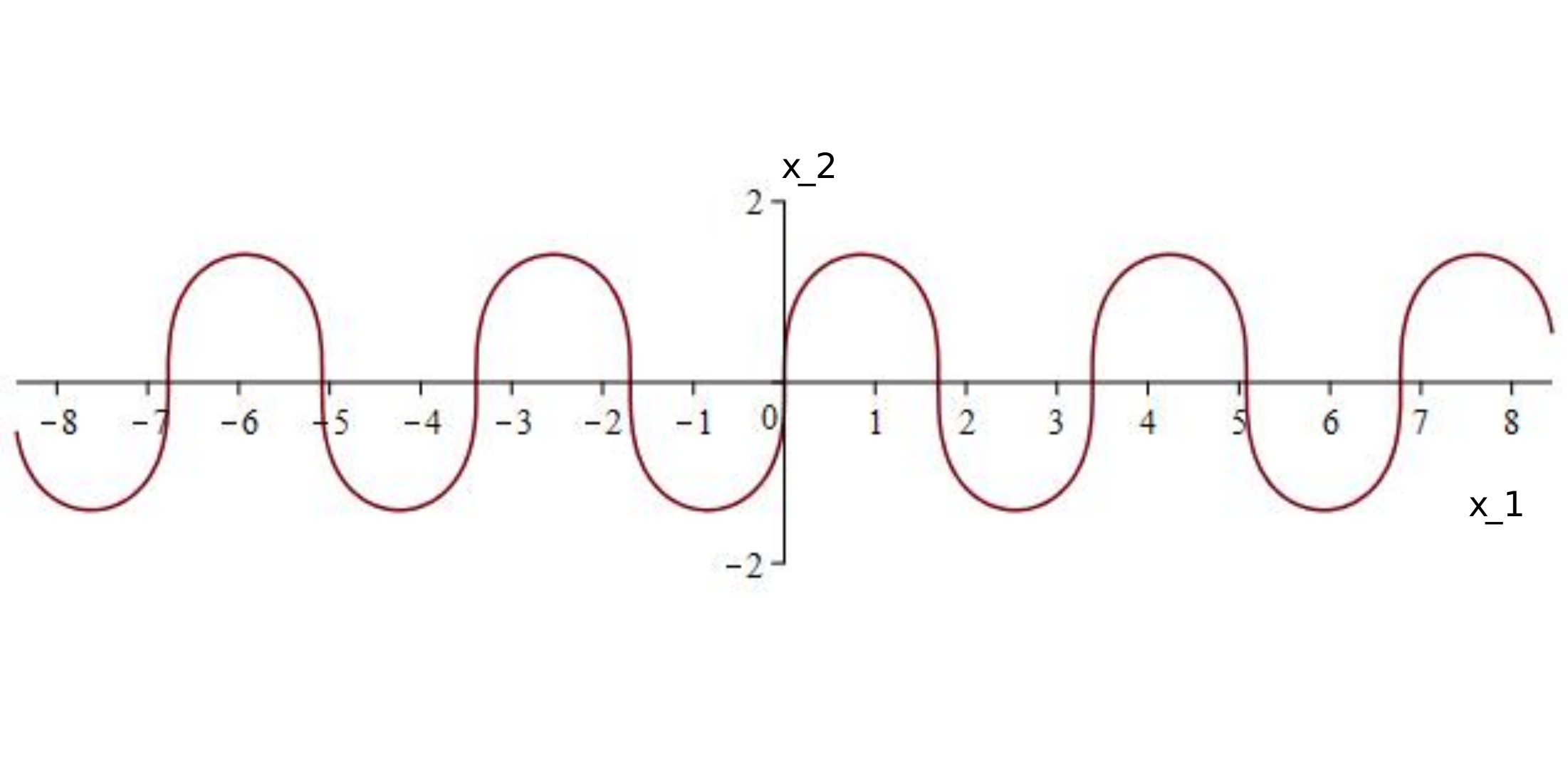} 
\end{center} 
 \caption{The Willmore curve $\gamma_T$}
\end{figure}
Notice that, by the vanishing of Gaussian curvature of cylindrical surfaces,  for one-dimensional curves the Willmore equation reduces to an ODE for the planar curvature, namely 
$$
 k'' = - \frac{1}{2} k^3. 
$$
This equation can be  explicitly solved using special functions, and then 
integrated to produce the above Willmore curves $\gamma_T$. Indeed, every non-affine complete planar Willmore 
curve coincides, up to an affine transformation with the curve $\gamma_T$, see \cite{Ma}.

Apart from producing a first non-compact profile of this type for the equation, our aim is to 
explore the relation between one-dimensionality of solutions and their limit properties. 
In fact, our construction shows that the straightforward counterpart of Gibbons' conjecture 
for \eqref{allen-cahn} is false. Our main result reads as follows. 

\begin{theorem}\label{t:ex}
There exists $T_0>0$ such that, for any $T>T_0$ there exists a $T$-periodic 
planar Willmore curve $\gamma_{T}$ and a solution $u_T$ to 
  \begin{equation*}
 -\Delta(-\Delta u_T+W'(u_T))+W^{''}(u_T)(-\Delta u_T+W'(u_T))=0,\label{CH_mult}
 \end{equation*}
such that  
$$
   u_T(x_1,x_2) \to \pm 1 \quad \hbox{ as } x_2 \to \pm \infty \qquad \hbox{ uniformly for } x_1 \in \R, 
$$
and 
\begin{equation}\label{eq:dist}
dist(\gamma_T,\{x\in\R^2:u_T(x)=0\})< \frac{c}{T}. 
\end{equation}
The function $u_T$ also satisfies the symmetries
$$
 u_T(x_1,x_2) = - u_T(-x_1,-x_2)= - u_T \left( x_1+\frac{L}{2},-x_2\right) \qquad \quad \hbox{ for every } x \in \R^2. 
$$ 
In particular, it is $L$-periodic in the $x_1$ variable, where $L:=(\gamma_T)_1(T)-(\gamma_T)_1(0)>0$, and 
furthermore, there exists a fixed constant $C_0$ such that
\begin{equation}\label{eq:approx-monot}
  \partial_{x_2} u_T(x_1,x_2)  \geq - C_0 \, T^{-3} \qquad \quad \hbox{ for all } (x_1, x_2) \in \R^2 \hbox{ and all } T > T_0.
\end{equation}
\end{theorem}

\


In the literature there are nowadays several constructions of interfaces starting from 
given limit profiles via Lyapunov-Schmidt reductions, see the above-mentioned references. However, 
being the Cahn-Hilliard of fourth order, here one needs a rather careful expansion using also 
smoothing operators. Moreover, to our knowledge, our solution seems to be the first one in the literature with a non-compact 
(and non-trivial) transition profile for \eqref{Cahn-Hilliard}. 

Notice also that the curve $\gamma_T$ 
is vertical at an equally-spaced sequence of points lying on the $x$-axis. Therefore the 
gradient of $u_T$ is nearly horizontal at these points, and it is quite difficult to 
understand the monotonicity (in $x_2$) of the solutions in these regions. Apart from the fact that the 
equation is of fourth-order, and hence rather involved to analyse, we need to expand  formally \eqref{Cahn-Hilliard} up to 
the fifth order  in $\frac{1}{T}$ for proving the estimate \eqref{eq:approx-monot}. 
In practice, we need to find a sufficiently good approximate solution to \eqref{Cahn-Hilliard} 
by adding suitable corrections to a naive transition layer along $\gamma_T$, and then 
by tilting properly the transition profile by a $T$-periodic function $\phi$. This tilting, which is of order $O(T^{-1})$, 
satisfies a {\em linearized Willmore equation} of the form 
$$
  \tilde{L}_0 \, \overline{\phi} = \bar{g}, 
$$
where $\bar{g}(t)$ is an explicit function of the curvature of $\gamma_T$ and its derivatives. 
The special structure of the 
right-hand side in our case and the \emph{special structure} of the Willmore equation 
make it possible to find an explicit solution (again, in terms of special functions) for $\overline{\phi}$, 
depending only on the curvature of $\gamma_T$ and its derivatives.

\begin{remark}
Unfortunately the main order term $\overline{\phi}$ in the perpendicular tilting of the interface with respect to 
$\gamma_T$ is flat at its vertical points,  so we can neither claim a full monotonicity of the solutions, nor 
disprove it. With our analysis and some extra work it should be possible to prove  monotonicity of $u_T$ in suitable portions of the plane, 
however to understand the monotonicity near those special points one would need either 
much more involved expansions and/or different ideas. 
\end{remark}


\

The plan of the paper is the following. In Section \ref{s:pwc} we study planar Willmore curves, and 
analyse some properties, including the spectral ones, of the linearised Willmore equation. In Section 
\ref{s:app-sol} we construct approximate solutions, expanding \eqref{Cahn-Hilliard} up to 
the fifth order in $T^{-1}$, in order to understand the normal tilting of the interface to $\gamma_T$. 
In Section \ref{s:red} we give the outline of the proof of our main result, performing a 
Lyapunov-Schmidt reduction of the problem on the normal tilting $\phi$. 
Sections \ref{s:tech} and the appendix are devoted to the proofs of some technical results: 
the former, concerning the reduction technique, while the latter dealing with the 
main order term $\overline{\phi}$ in the expansion of $\phi$.

\begin{center}
{\bf Acknowledgements} 
\end{center}

A.M. is supported by the project {\em Geometric Variational Problems} from Scuola Normale Superiore and 
by MIUR Bando PRIN 2015 2015KB9WPT$_{001}$.  He is also member of GNAMPA as part of INdAM. 
R.M.  would like to thank the Deutsche Forschungsgemeinschaft (DFG, German Research Foundation)
for the financial support via the grant {MA~6290/2-1}. 
R. M. and M. R. would like to thank Scuola Normale Superiore for the kind hospitality during the preparation 
of this manuscript.

\section{Planar Willmore curves}\label{s:pwc}

In this section we collect some material about existence of planar Willmore 
curves, analysing then their spectral properties with respect to the second variation 
of the Willmore energy. 
Recall that the Willmore energy of a curve $\gamma : [0,1] \to \R^2$ is defined as  the 
integral of the curvature squared 
$$
  \int_{\gamma} k(s)^2 ds. 
$$
Extremizing with respect to variations that are compactly supported in $(0,1)$ one finds 
that critical points satisfy the Willmore equation 
\begin{equation} \label{eq:Will}
 k'' = - \frac{1}{2 } k^3. 
\end{equation}

\subsection{Existence of Willmore curves} 

Recall first the definition of the Jacobi cosine function, see for example \cite{BW}. For $m \in (0,1)$ define 
$$
  \sigma(\varphi, m) = \int_0^\varphi \frac{d \theta}{\sqrt{1-m \sin^2 \theta}}, 
$$
and then implicitly the function $\mathrm{cn}$ by 
\begin{equation}
 \mathrm{cn } \, \left( \sigma(\varphi, m) | m \right) = \cos \varphi. 
\end{equation}
Equation \eqref{eq:Will} admits (only) periodic solutions  that, up to 
a   dilation and translation are given by 
\begin{equation}\label{eq:ks}
k(s)  =  \sqrt{2} \, \text{cn} \left(\left.s+\bar{T}/4 \, \right| \, \frac{1}{2} \right). 
\end{equation}
For this choice, the period $\bar{T}$ has the approximate value $\bar{T} \simeq 7.416$. 
Using the conservation of Hamiltonian energy, 
this function satisfies 
\begin{equation} \label{eq:cons}
 (k')^2(s)=-\frac{k^4(s)}{4}+1.
\end{equation}
The above function can be integrated to produce a Willmore curve, by the formula 
(with an abuse of notation, we will always use the same letter both for the curve and for its parametrization) 
\begin{eqnarray*}
  \gamma(s)=  \int_0^s \begin{pmatrix}
    -\sin(\int_0^s k(\tau) d\tau) \\
    \cos(\int_0^s k(\tau) d\tau)
  \end{pmatrix}ds. 
\end{eqnarray*}
Notice that $\gamma$ is parametrized by arc length. If we set
$\gamma_T(s):=\varepsilon^{-1}\gamma(\varepsilon s)$, for $\eps=\bar{T}/T$, then it is still
true that $|\gamma^{'}_T(s)|=1$ for any $s\in\R$. In other words, $\gamma_T$ also
denotes the rescaled curve $\{\varepsilon^{-1}\zeta:\zeta\in\gamma\}$, still parametrized by arc length. 
Our aim is  to construct solutions $u_T$ with a transition layer close to $\gamma_T$, that are odd and periodic in $x_1$ and 
fulfilling the symmetry property
\begin{eqnarray}
  u_T(x_1,x_2)=-u_T(-x_1,-x_2)=-u_T\left(x_1+\frac{L}{2},-x_2\right), & L:=(\gamma_T)_1(T)-(\gamma_T)_1(0).
\label{symm}
\end{eqnarray}
The curvature of $\gamma_T$ is defined by
$$
k_\eps(s):=-\langle\gamma_T^{'}(s),\gamma_T^{''}(s)^\bot \rangle, \qquad w^\bot=(-w_2,w_1), 
$$
and clearly by the arc-length parametrisation one has $\gamma_T^{''}(s)=k_{\varepsilon}(s)\gamma_T^{'}(s)^\bot$. 
In what follows, when the subscript $\eps$ is omitted, it will be assumed to be equal to $1$,  i.e. we will set $\gamma:=\gamma_1$, $k:=k_{1}$, 
etc..

\subsection{The linearized problem}

We discuss next the linearization of the Willmore equation, namely we consider 
the  problem
\begin{equation} 
\tilde{L}_{0} \, \phi=g \qquad \text{in }\R,
\label{bifo_lin}
\end{equation}
where $g:\R\to\R$ is a given ${\bar{T}}-$periodic function. Recall from formula (33) in \cite{LMS} that $\tilde L_0$ is given by
\begin{equation}\label{eq:tL0}
 \tilde L_0 \, \phi =\phi^{(4)} + (\frac{5}{2}k^2\phi^{'})' + (3(k')^2-\frac{1}{2}k^4)\phi = \phi^{(4)} + (\frac{5}{2}k^2\phi^{'})' + (3-\frac{5}{4}k^4)\phi,
\end{equation}
where the conservation law \eqref{eq:cons} has been used.  Given the symmetries of the problem, we are interested 
in right-hand sides $g$ that satisfy the following conditions 
\begin{equation*}
g(s)=-g(-s)=-g(s+{\bar{T}}/2), 
\end{equation*}
hence we define the spaces
\begin{equation}
C^{n,\alpha}_{{\bar{T}}}(\R):=\bigg\{\phi\in C^{n,\alpha}(\R):\phi(s)=-\phi(-s)=-\phi\bigg(s+\frac{{\bar{T}}}{2}\bigg)\bigg\}\label{def_per_spaces},
\end{equation}
where ${\bar{T}}>0$, $n\geq 0$ is an integer, $0<\alpha<1$ and $C^{n,\alpha}(\R)$ is the space of functions $\phi:\R\to\R$ that are $n$ times differentiable and whose $n$-th derivative is H\"{o}lder continuous of exponent $\alpha$. We endow the spaces $C^{n,\alpha}_{{\bar{T}}}(\R)$ with the norms 
$$
||\phi||_{C^{n,\alpha}(\R)}=\sum_{j=0}^{n}||\nabla^{j}\phi||_{L^{\infty}(\R)}+\sup_{s\neq t}\frac{|\phi^{(n)}(s)-\phi^{(n)}(t)|}{|t-s|^{\alpha}}.
$$
Roughly speaking, these spaces consist of functions that respect the symmetries of the curve $\gamma$, in the sense that they are even, periodic with period ${\bar{T}}$, and they change sign after a translation of half a period.
We have then the following result. 

\begin{proposition} \label{prop_lin_bifo}
  Let ${\bar{T}}>0$. Let $g\in C^{0,\alpha}(\R)$ satisfy $g(s)=-g(-s)=-g(s+{\bar{T}}/2)$ for all $s\in\R$. Then there is a
  unique function $\phi\in C^{4,\alpha}(\R)$ that solves equation \eqref{bifo_lin} and satisfies
$$
    \phi(s)=-\phi(-s)=-\phi(s+\bar{T}/2)\;\qquad \forall s\in\R. 
$$
  Moreover, the estimate $||\phi||_{C^{4,\alpha}(\R)}\leq c||g||_{C^{0,\alpha}(\R)}$ holds for some positive
  number $c$ independent of $g$.
\end{proposition}
\begin{proof}
By considering extensions by periodicity, it is sufficient to prove unique resolvability of $\tilde L_0 \, \phi=g$ on $[0,{\bar{T}}]$ for $\phi$ in the space
\begin{eqnarray}\notag
C^{4,\alpha}_{{\bar{T}},0}([0,{\bar{T}}]):=\{\phi|_{[0,{\bar{T}}]}:\phi\in C^{4,\alpha}_{{\bar{T}}}(\R)\}\\\notag
=\{\phi\in C^{4,\alpha}([0,{\bar{T}}]):\phi(s)=-\phi({\bar{T}}-s)=-\phi(s+{\bar{T}}/2)\; \hbox{ for } \;  0\leq s\leq {\bar{T}}/2\}.
\end{eqnarray}
We observe that, by construction, any function $\phi\in C^{4,\alpha}_{{\bar{T}},0}([0,{\bar{T}}])$ satisfies 
$\phi^{(j)}(0)=\phi^{(j)}({\bar{T}})$, $0\leq j\leq 4$, hence we can extend $\phi$ to a function in $C^{4,\alpha}_{{\bar{T}}}(\R)$. 

We denote by $(-\Delta)^{-2}h$ the unique solution $\phi$ to 
\begin{equation}\notag
\begin{cases}
\phi^{(4)}=h &\text{on }\quad [0,{\bar{T}}];\\\notag
\phi(0)=\phi({\bar{T}})=\phi^{''}(0)=\phi^{''}({\bar{T}})=0 &\text{(homogeneous Navier boundary conditions),}
\end{cases}
\end{equation}
where $h\in C^{0,\alpha}_{{\bar{T}}}(\R)$ is given. Such a solution $\phi$ is in $C^{4,\alpha}([0,{\bar{T}}])$ and fulfils the estimate
 $$
    ||\phi||_{C^{4,\alpha}(\R)}
    = ||\phi||_{C^{4,\alpha}(0,{\bar{T}})}
    \leq c\, ||h||_{C^{0,\alpha}(0,{\bar{T}})}
    = c\, ||h||_{C^{0,\alpha}(\R)}. 
  $$ 
Since $h$ verifies the symmetries $h(s)=-h({\bar{T}}-s)=-h(s+{\bar{T}}/2)$, for any $0\leq s\leq {\bar{T}}/2$, then also does $\phi$, thus $\phi\in C^{4,\alpha}_{{\bar{T}},0}([0,{\bar{T}}])$.\\

Then $\tilde L_0 \, \phi=g$ is
  equivalent to 
  $$
    \phi + (-\Delta)^{-2} \Big( (\frac{5}{2}k^2\phi ')' + (3-\frac{5}{4}k^4)\phi\Big) = (-\Delta)^{-2}(g).
  $$
  In particular the Fredholm alternative in the space $C^{4,\alpha}_{{\bar{T}},0}([0,{\bar{T}}])$ applies, so  the equation is
  solvable for every $g\in C^{0,\alpha}_{{\bar{T}}}(\R)$ if and only if the homogeneous problem is uniquely solvable.

Exploiting the symmetries of the Willmore equation, if $\nu_{\gamma}$ denotes the 
normal vector to the curve $\gamma$, the following four functions represent 
Jacobi fields for the linearized Willmore equation 
$$
   \psi_1 = \langle (0,-1), \nu_{\gamma} \rangle; \qquad 
   \psi_2 = \langle (1,0), \nu_{\gamma} \rangle; \qquad 
   \psi_3 = \langle (\gamma_2,-\gamma_1), \nu_{\gamma} \rangle ; 
   \qquad \psi_4 = \langle (\gamma_1,\gamma_2), \nu_{\gamma} \rangle. 
$$
Using the fact that $\nu_{\gamma} = (\gamma'_2, - \gamma'_1)$  (here $\gamma_1$ and 
$\gamma_2$ represent the horizontal and vertical components of $\gamma$) 
 one finds that 
 $$
 (\psi_1,\ldots,\psi_4):=(\gamma_1',\gamma_2',\gamma_1\gamma_1'+\gamma_2\gamma_2',
    \gamma_1\gamma_2'-\gamma_2\gamma_1'). 
  $$
We claim that these functions are linearly independent: indeed, using $k'(0)=k'({\bar{T}})=-1$ and $\gamma_1({\bar{T}})\neq 0$ we get
  $$
    \begin{pmatrix}
      \psi_1(0) \\ \psi_1({\bar{T}}) \\ \psi_1^{''}(0) \\\psi_1^{''}({\bar{T}})
    \end{pmatrix}
    = \begin{pmatrix}
      0 \\ 0 \\ 1 \\ 1
    \end{pmatrix},\qquad \qquad 
    \begin{pmatrix}
      \psi_2(0) \\ \psi_2({\bar{T}}) \\ \psi_2^{''}(0) \\\psi_2^{''}({\bar{T}})
    \end{pmatrix}
    = \begin{pmatrix}
      1 \\ 1 \\ 0 \\ 0
    \end{pmatrix}, 
    $$
    $$
    \begin{pmatrix}
      \psi_3(0) \\ \psi_3({\bar{T}}) \\ \psi_3^{''}(0) \\\psi_3^{''}({\bar{T}})
    \end{pmatrix}
    = \begin{pmatrix}
      0 \\ 0 \\ 0 \\ \gamma_1({\bar{T}})
    \end{pmatrix},\qquad \qquad 
    \begin{pmatrix}
      \psi_4(0) \\ \psi_4({\bar{T}}) \\ \psi_4^{''}(0) \\\psi_4^{''}({\bar{T}})
    \end{pmatrix}
    = \begin{pmatrix}
      0 \\ \gamma_1({\bar{T}}) \\ 0 \\ 0
    \end{pmatrix}. 
  $$ 
Being the   homogeneous ODE $\tilde L_0\phi=0$ of fourth order in $\phi$, 
all its $\bar{T}$-periodic solutions are spanned by $\{ \psi_1, \psi_2, \psi_3, \psi_4 \}$. 
From the above formulas one infers that a function $\phi\in C^{4,\alpha}_{{\bar{T}},0}([0,{\bar{T}}])$ that is a linear combination 
$(\psi_1, \psi_2, \psi_3, \psi_4 )$ satisfies homogeneous Navier boundary conditions if and only if  it is trivial. 
  Hence the homogeneous problem  has only the trivial solution, and the equation $\tilde L_0 \, \phi=g$ 
  (with the desired boundary conditions) is uniquely
  solvable in $C^{4,\alpha}_{{\bar{T}},0}([0,{\bar{T}}])$, as claimed. The norm estimate follows from  
  $$
    ||\phi||_{C^{4,\alpha}(\R)}
    = ||\phi||_{C^{4,\alpha}(0,{\bar{T}})}
    \leq c||g||_{C^{0,\alpha}(0,{\bar{T}})}
    = c||g||_{C^{0,\alpha}(\R)}, 
  $$ 
  where the inequality results from higher order Schauder estimates, see e.g. \cite{GrG}. 
  Notice that reducing the problem on $\R$ to a problem on $[0,{\bar{T}}]$ ensures compactness.
\end{proof}

\

We need next  to invert the linearized operator for a specific right-hand side, arising from high-order 
expansion (in $\eps = \frac{\bar{T}}{T}$) of the approximate solutions, see Section \ref{s:app-sol}. We have the following result.

\begin{proposition}\label{p:inv-rhs}
    Let
$$
  \overline{g}(s) = \frac{9}{8} k(s)^5 - 9 k(s) k'(s)^2. 
$$
   Then the equation $\tilde{L}_0 \, \bar{\phi}=\overline{g}$ admits a unique smooth
    solution $\bar\phi\in C^{4,\alpha}_{{\bar{T}}}(\R)$ which additionally  satisfies
    \begin{itemize}
      \item[(i)] $\bar\phi(s) \, k(s) \geq  0$ for all $s\in\R$, 
      \item[(ii)] $\bar\phi'(s)=0$ whenever $k(s)=0$.
    \end{itemize}     
  \end{proposition}

\

\begin{remark}
  The solution $\bar\phi$ can be written  explicitly in terms of hyper-geometric functions, see Chapter 15
  in \cite{AS} or \cite{AA} for the notation we are using and additional properties. Indeed, for every
  $\mu_0,\mu_1\in\R$ a formal solution $\bar{\phi}$ is given by the formula $\bar{\phi}(s) = \Phi(k(s))$, where
  $$
    \Phi(z) := \mu_0z + \mu_1z^3+\frac{37-40\mu_0}{960} z^5+r_1(z)+r_2(z)+r_3(z), 
  $$
  for functions $r_1,r_2,r_3$ given by
$$
    r_1(z) := \frac{\mu_1}{28} z^7 \, \,_2F_1\left( 1,\frac{5}{4}; \frac{11}{4}; \frac{z^4}{4}\right),  \qquad \quad 
    r_2(z) := - \frac{41}{640} z^9 \, \,_2F_1\left(1, \frac{7}{4}; \frac{13}{4}; \frac{z^4}{4}\right), 
    $$
    $$
    r_3(z) :=  \left( -\frac{3}{896}\mu_0+ \frac{465}{7168}\right) z^9 \, \,_3F_2\left(1, \frac{7}{4},\frac{5}{2};
    \frac{11}{4},3; \frac{z^4}{4}\right). 
  $$ 
  This representation, however, does not seem to be helpful when discussing the regularity properties of
  the function $\Phi \circ k$ as we will discuss in the proof below.
\end{remark}

  \begin{proof}
    Since we are looking for an odd solution $\bar\phi\in C^{4,\alpha}_{\bar{T}}(\R)$, 
    motivated by the special features of Willmore's equation we 
    consider the ansatz
    $\bar\phi(s)=\Phi(k(s))$. After some calculations
    it is possible to write $\tilde{L}_0 \, \bar{\phi}=\overline{g}$ as an ODE for $\Phi$, namely
    \begin{align} \label{eq:ODE_for_Phi}
      \begin{aligned}
 	  & \frac{1}{16} \Big( (z^4-4)^2 \Phi^{(4)}(z)+12 z^3(z^4-4)  \Phi^{(3)}(z)+ 2 z^2 (13 z^4 - 28)  \Phi''(z)  \\
  	  &\quad -16 (z^4-2) z \Phi'(z)+(48-20 z^4) \Phi(z)\Big) = -9z + \frac{27}{8}z^5.
  	  \end{aligned}
    \end{align}
    This equation can be solved explicitly
    in terms of a series $\Phi(z)= \sum_{k=0}^\infty \mu_k z^{2k+1}$, where the parameters $\mu_0,\mu_1$ are free and
    $\mu_2,\mu_3,\ldots$ are  determined recursively by the above ODE. Their precise definition are provided in the
    Appendix, see \eqref{eq:defn_cm}. This series has convergence radius
    $\sqrt{2}$ so  it is not clear a priori whether $s\mapsto \bar\phi(s)=\Phi(k(s))$ defines a function 
    of $C^{4,\alpha}_{\bar{T}}(\R)$. In order to ensure this we impose that the solution $\bar\phi$ is even about $-\bar{T}/4$ and $\bar{T}/4$,
    i.e.  we require
	\begin{equation} \label{eq:oddness_barphi}
	  \bar\phi '(s)\to 0,\quad \bar\phi'''(s)\to 0\quad\text{as } |s|\to \bar{T}/4.
	\end{equation}
	Notice that $k(s)$ converges to $\sqrt{2}$ as $|s|\to \frac{T}{4}$. The
	calculations from the Appendix show that \eqref{eq:oddness_barphi} holds if and only if we choose
    \begin{equation} \label{eq:c0c1}
      \mu_0 = 0,\qquad \qquad \qquad \mu_1= \frac{\pi^2}{8\Gamma(\frac{3}{4})^4}.
    \end{equation}
    The corresponding solution is given by
    \begin{align*}
      \bar\phi(s) 
      &=   \frac{3\pi\sqrt{2}}{64\Gamma(\frac{3}{4})^2}\cdot
      \sum_{m=0}^\infty \Big( -  \frac{\Gamma(m+\frac{3}{4})}{2\cdot 4^m\Gamma(m+\frac{9}{4})}  
      k(s)^{4m+5} + \frac{ \Gamma(m+\frac{1}{4})}{4^m\Gamma(m+\frac{7}{4})} k(s)^{4m+3} \Big).
	\end{align*}
    Thanks to \eqref{eq:oddness_barphi} this solution can be reflected evenly about $s=\pm \bar{T}/4$, so 
    we obtain by standard arguments that $\bar\phi\in C^{4,\alpha}_{\bar{T}}(\R)$. From the above formula we find $\bar\phi'(s)= O(k'(s)k(s)^2)\to 0$ as
    $k(s)\to 0$ as well as
    \begin{align*}
      \bar\phi(s)k(s) 
      &=   \frac{3\pi\sqrt{2}k(s)^4}{64\Gamma(\frac{3}{4})^2}\cdot
      \sum_{m=0}^\infty  \Big(
      -\frac{\Gamma(m+\frac{3}{4})}{\Gamma(m+\frac{9}{4})} \cdot \frac{k(s)^2}{2} +
      \frac{\Gamma(m+\frac{1}{4})}{\Gamma(m+\frac{7}{4})}   \Big) \Big(\frac{k(s)^4}{4}\Big)^m\\
      &\geq  \frac{3\pi\sqrt{2}k(s)^4}{64\Gamma(\frac{3}{4})^2}\cdot
      \sum_{m=0}^\infty  \underbrace{\Big( -\frac{\Gamma(m+\frac{3}{4})}{\Gamma(m+\frac{9}{4})} +
      \frac{\Gamma(m+\frac{1}{4})}{\Gamma(m+\frac{7}{4})}   \Big)}_{\geq 0}\Big(\frac{k(s)^4}{4}\Big)^m \geq
      0.
	\end{align*}
	Hence, claim (i) and (ii) are proved and we can conclude.
  \end{proof}

\section{Approximate solutions}\label{s:app-sol}

In this section we introduce an approximate solution of $\eqref{eq:W}$, which we need to expand up to the fifth order in $\eps=\bar{T}/T$. For doing this, we use Fermi coordinates around a perturbation of the curve 
$$
  \gamma_{T}(\cdot) = \frac{1}{\eps} \gamma \left( \eps \, \cdot \right); \qquad \quad \eps = \frac{\bar{T}}{T}, 
$$
($\gamma$ is the Willmore curve constructed in Section \ref{s:pwc}) and we expand both the Laplace operator and Cahn-Hilliard equation. We also need to add suitable corrections to the approximate solution in order to improve its accuracy: these will allow us to 
study in more detail the {\em transition curve} $\{ u_T = 0 \}$ of the solution constructed in Theorem \ref{t:ex}.

\subsection{Fermi coordinates near $\gamma_T$}

As in \cite{Ri}, we want to use Fermi coordinates near a normal perturbation of the dilated periodic curve $\gamma_T$. To this end, we fix $\phi\in C^{4,\alpha}_{\bar{T}}(\R)$ 
(recall \eqref{def_per_spaces}) such that $\|\phi\|_{C^{4,\alpha}(\R)}< 1$ , and for $T$ large (i.e. for $\eps = \tfrac{\bar{T}}{T}$ small) we define the planar map  
\begin{equation}
Z_\eps(s,t) = \gamma_T(s) + (t+\phi(\eps s)) \gamma_T^{'}(s)^\perp, \qquad w^\perp := (-w_2,w_1).
\label{def_Z}
\end{equation}
Using the fact that $\gamma_T:=\eps^{-1}\gamma(\eps s)$ and $|\gamma_T^{'}|=1$ we find
  \begin{align*}
    \det(\partial_s Z_\eps,\partial_t Z_\eps)
    &= \det\Big(\gamma^{'}(\eps s)^\perp, \gamma^{'}(\eps s) + \eps\phi^{'}(\eps s) \gamma^{'}(\eps s)^\perp \pm
    (t+\phi(\eps s))\eps k(\eps s) \gamma^{'}(\eps s)
     \Big) \\
    &= \det(\gamma^{'}(\eps s)^\perp, \gamma^{'}(\eps s)) \cdot \big(1+\eps k(\eps s)(t+\phi(\eps s))\big) \\
    &\geq 1 - \sqrt{2}(|t|+1)\eps \\
    &\geq \frac{1}{4} \qquad \qquad\text{for } |t|< \frac{1}{2\sqrt{2}\eps}. 
  \end{align*}
This shows that in the above region the map  $Z_\eps$ is invertible.  
  Moreover, define 
$$ 
    V_{\varepsilon,\phi} :=\left\{x\in\R^{2}:dist(x,\gamma_{T,\phi})< \frac{1}{4\eps}\right\},
  $$  
and $\gamma_{T,\phi}:=\gamma_T(s)+\phi(\eps s)\gamma_T^{'}(s)^\bot$.  
Since $\phi\in C^{4,\alpha}(\R)$ it follows also that $Z_\eps:\R\times (-\frac{1}{4\eps},\frac{1}{4\eps})\to V_{\eps,\phi}$ is a $C^{4,\alpha}$-diffeomorphism.  
 With an abuse of notation, we will write $u$ for $u(s,t)$. We also 
set
\begin{equation}\label{eq:Veps}
V_{\varepsilon} :=\left\{x\in\R^{2}:dist(x,\gamma_{T})< \frac{1}{4\eps}\right\},
\end{equation}

\subsection{The Laplacian in Fermi coordinates}

We are interested in the expression of the Laplacian in the above coordinates $(s,t)$. First we assume that $\phi=0$, that is we consider the diffeomorphism
$$
\tilde{Z}_\eps:\R\times(-1/4\eps,1/4\eps)\to V_\eps
$$
defined by
\begin{equation}\label{eq:sz}
\tilde{Z}_\eps(s,z):=\gamma_T(s)+z\gamma_T^{'}(s)^\bot. 
\end{equation}
The euclidean metric in these coordinates is
\begin{eqnarray}\notag
g=
\begin{bmatrix}
(1-\eps zk(\eps s))^2 & 0\\
0 & 1
\end{bmatrix},
\end{eqnarray}
with determinant  $\det g=g_{ss}=(1-\eps zk(\eps s))^2$ and  inverse given by 
\begin{eqnarray}\notag
g^{-1}=
\begin{bmatrix}
(1-\eps zk(\eps s))^{-2} & 0\\
0 & 1
\end{bmatrix}.
\end{eqnarray}
From now on, the curvature $k$ and its derivatives will always be evaluated at $\eps s$. Using that $g^{ss}=g_{ss}^{-1}$, the Laplacian with respect to this metric is given by
\begin{eqnarray}\label{eq:lapl} \nonumber
\Delta u=\frac{1}{\sqrt{g_{ss}}}\partial_s(\sqrt{g_{ss}}g^{ss}\partial_s u)+\frac{1}{\sqrt{g_{ss}}}\partial_z(\sqrt{g_{ss}}\partial_z u)\\
=\frac{1}{\sqrt{g_{ss}}}\partial_s(\frac{1}{\sqrt{g_{ss}}}\partial_s u)+\partial^2_z u+\frac{\partial_z g_{ss}}{2g_{ss}}\partial_z u\\ \nonumber
=\frac{\partial^2_s u}{g_{ss}}-\frac{\partial_s g_{ss}}{2g_{ss}^2}\partial_s u+\partial^2_z u+\frac{\partial_z g_{ss}}{2g_{ss}}\partial_z u.
\end{eqnarray}
Now we compute
\begin{eqnarray}\notag
\frac{\partial_z g_{ss}}{2g_{ss}}=-\eps\frac{k}{1-\eps zk}.
\end{eqnarray}
Taylor expanding in $\eps$, we get
\begin{eqnarray}\notag
\frac{k}{1-\eps zk}=k+\eps k^2 z+\eps^2 k^3 z^2+\eps^3 k^4 z^3+\eps^4 k^5 z^4+\eps^5\overline{h}(\eps s,z),
\end{eqnarray}
where the remainder term $\overline{h}(\eps s,z)$ satisfies 
\begin{eqnarray}\notag
\partial^{(i)}_s \overline{h}(\eps s,z)=O(\eps^i z^{i+5}); \qquad \quad i \geq 0. 
\end{eqnarray}
Therefore, using the same notation $\overline{h}(\eps s,z)$ for a remainder term similar to the previous one we also have 
\begin{eqnarray}\notag
\frac{\partial_z g_{ss}}{2g_{ss}}=-\eps k-\eps^2 k^2z-\eps^3 k^3 z^2-\eps^4 k^4z^3-\eps^5 k^5z^4+\eps^6\overline{h}(\eps s,z),
\end{eqnarray}
Now we Taylor-expand in $\eps$ the following quantities 
\begin{eqnarray}\notag
\frac{1}{g_{ss}}=\frac{1}{(1-\eps zk)^2}=1+2\eps zk+3\eps^2z^2 k^2+\eps^3\overline{a}(\eps s,z); \\\notag
-\frac{\partial_s g_{ss}}{2g_{ss}^2}=\frac{\eps^2 zk'}{(1-\eps zk)^3}=\eps^2 zk'(1+3\eps zk+\eps^2\overline{b}(\eps s,z)),
\end{eqnarray}
where the remainders $\overline{a}(\eps s,z)$, $\overline{b}(\eps s,z)$ satisfy 
\begin{eqnarray}\notag
\partial^{(i)}_s \overline{a}(\eps s,z)=O(\eps^i z^{i+3}), \qquad \quad \partial^{(i)}_s \overline{b}(\eps s,z)=O(\eps^i z^{i+2}), 
\quad i \geq 0. 
\end{eqnarray}
In conclusion, the expansion of the Laplacian in the above coordinates $(s,z)$ (see \eqref{eq:sz}) is 
\begin{eqnarray}
\Delta=\partial^2_z+\partial^2_s-\eps k\partial_z-\eps^2 k^2 z\partial_z-\eps^3 k^3 z^2\partial_z-\eps^4 k^4 z^3\partial_z-\eps^5 k^5 z^4\partial_z+\eps^6\overline{h}(\eps s,z)\label{exp_lapl_phi=0}\\\notag
+\eps z(2k\partial^2_s+\eps k'\partial_s)+\eps^2 z^2(3k^2\partial^2_s+3\eps kk'\partial_s)
+\eps^3(\overline{a}(\eps s,z)\partial^2_s +\eps zk'\overline{b}(\eps s,z)\partial_s ),
\end{eqnarray}
where, we recall, $k$ and its derivatives are evaluated at $\eps s$.

Given a function
\begin{eqnarray}\notag
f:\R^2\to\R
\end{eqnarray}
of class $C^{2}$, it is possible to make the change of variables $t:=z-\phi(\eps s)$. In other words, we define
\begin{eqnarray}\notag
\tilde{f}:\R^2\to\R
\end{eqnarray}
by setting $\tilde{f}(s,z):=f(s,z-\phi(\varepsilon s))$. A straightforward computation shows that
\begin{eqnarray}\notag
\partial_z \tilde{f}(s,z)=\partial_t f(s,z-\phi);\\\notag
\partial_s \tilde{f}(s,z)=\partial_s f(y,z-\phi)-\varepsilon\phi^{'}\partial_t f(s,z-\phi);\\\notag
\partial^2_s \tilde{f}(s,z)=\partial^2_s f(y,z-\phi)-2\varepsilon\phi^{'}\partial_{st}f(y,z-\phi);\\\notag
-\varepsilon^{2}\phi^{''}\partial_t f(y,z-\phi)+\varepsilon(\phi^{'})^2\partial^2_t f(y,z-\phi),
\end{eqnarray}
where, we recall, $\phi$ and its derivatives are evaluated at $\varepsilon s$. Hence by \eqref{eq:lapl} the expansion we are interested in, using the 
latter coordinates $(s,t)$,  is given by 
\begin{eqnarray}
\Delta=\partial^{2}_{s}+\partial^{2}_{t}-\varepsilon k\partial_{t}-\eps^{2}(t+\phi)k^2\partial_{t}-\eps^{3}(t+\phi)^2 k^3\partial_{t}-\eps^{4}(t+\phi)^3 k^4\partial_{t}\label{exp_lapl}\\\notag
-\eps^{5}(t+\phi)^4 k^5\partial_{t}-\eps^6\overline{h}\partial_t-\varepsilon^{2}\phi^{''}\partial_{t}-2\varepsilon\phi^{'}\partial_{st}
+\varepsilon^{2}(\phi^{'})^{2}\partial^{2}_{t}\\\notag
+\varepsilon(t+\phi)\{2k\partial^{2}_{s}+\varepsilon k^{'}\partial_{s}-\varepsilon^{2}(2k\phi^{''}
+k^{'}\phi^{'})\partial_t-4\varepsilon k\phi^{'}\partial_{st}+2\varepsilon^{2}k(\phi^{'})^{2}\partial^2_t\}\\\notag
+\varepsilon^{2}(t+\phi)^{2}\{3k^2\partial^{2}_{s}+3\varepsilon kk'\partial_{s}-\varepsilon^{2}(3k^2\phi^{''}
+3kk'\phi^{'})\partial_t-6\varepsilon k^2 \phi^{'}\partial_{st}+3\varepsilon^{2}k^2(\phi^{'})^{2}\partial^2_t\}\\\notag
+\eps^3\{\overline{a}(\eps s,t)(\partial^2_s-2\eps\phi^{'}\partial_{st}+\eps^2\phi^{''}\partial_t +\eps(\phi^{'})^2\partial^2_t)+\eps (t+\phi)k'\overline{b}(\eps s,t)(\partial_s-\eps\phi^{'}\partial_t)\},
\end{eqnarray}
see also formulas $(28)$ and $(29)$ in \cite{Ri}.

\subsection{Construction of the approximate solution}\label{ss:constr}

We proceed by fixing a function $\phi\in C^{4,\alpha}_{\bar{T}}(\R)$ such that $||\phi||_{C^{4,\alpha}(\R)}<1$ (recall \eqref{def_per_spaces}) and  constructing an approximate solution $v_{\varepsilon,\phi}$ whose nodal set is a perturbation of the initial
curve $\gamma_T$, tilting it transversally in the normal direction by $\phi$ (scaling properly its argument). 
This approximate solution is constructed in such a way that $v_{\varepsilon,\phi}\to\pm 1$ when the distance from $\gamma_T$ tends to infinity from different sides. More precisely, we observe that $\gamma_T$ divides $\R^{2}$ into two open unbounded regions: an \emph{upper part} $\gamma_T^{+}$ and a \emph{lower part} $\gamma_T^{-}$.

We set
$$
\mathbb{H}(x):=
\begin{cases}
1 &\text{if $x\in\gamma_T^{+}$}\\
0 &\text{if $x\in\gamma_T$}\\
-1 &\text{if $x\in\gamma_T^{-}$}
\end{cases}
$$
and introduce a $C^{\infty}$ cutoff function
$\zeta:\R\to\R$  such that
$$
\zeta(t)=
\begin{cases}
1 &\text{for $t<1$}\\
0 &\text{for $t>2$}.
\end{cases}
\label{def_zeta}
$$
For any $\varepsilon>0$ and for any integer $l>0$, recalling the definition of $V_\eps$ in \eqref{eq:Veps}, we set
\begin{equation}\label{eq:chil}
\chi_{l}(x):=
\begin{cases}
\zeta(|t|-\frac{1}{8\varepsilon}-l)&\text{ if }x=Z_{\varepsilon}(s,t)\in V_{\varepsilon},\\\notag
0&\text{ if }x\in\R^{2}\backslash V_{\varepsilon}. 
\end{cases}
\end{equation}
We will start by constructing an approximate solution $\hat{v}_{\varepsilon,\phi}$ in $V_{\varepsilon}$, and 
then {\em globalize} it using the above cut-off functions, introducing 
\begin{equation}
v_{\varepsilon,\phi}(x)=\chi_{5}(x)\hat{v}_{\varepsilon,\phi}(x)+(1-\chi_{5}(x))\mathbb{H}(x),  \qquad x\in\R^{2}.
\label{defv}
\end{equation}
Since $v_{\varepsilon,\phi}$ coincides with $\hat{v}_{\varepsilon,\phi}$ near $\gamma_T$, it is convenient to define $\hat{v}_{\varepsilon,\phi}$ through the Fermi coordinates $(s,t)$, see \eqref{def_Z}. In order to do so, we first define a function $\tilde{v}_{\varepsilon,\phi}(s,t)$ on $\R^{2}$, in such a way that its zero set is close to $\{ t = 0 \}$, then we set
$$
\hat{v}_{\varepsilon,\phi}(x):=
\begin{cases}
\tilde{v}_{\varepsilon,\phi}(Z_{\varepsilon}^{-1}(x)) &\text{if $x\in V_{\varepsilon}$}\\
0 &\text{if $x\in\R^{2}\backslash V_{\varepsilon}$.} 
\end{cases}
$$
In order to have a global definition of $v_{\varepsilon,\phi}$, the value of $\hat{v}_{\varepsilon,\phi}$ far 
from the curve is not relevant, since it is multiplied by a cut-off function that is identically zero there. We
stress that, by the symmetries of $\phi$, $v_{\varepsilon,\phi}$ satisfies the symmetry properties (\ref{symm}) 
if $\hat{v}_{\varepsilon,\phi}$ does.\\

A natural first guess for an approximate solution is $\tilde{v}_{\eps,\phi}(s,t):=v_0(t)$, where $v_{0}$ is the unique solution to the problem
$$
\begin{cases}
-v_{0}^{''}=v_0-v_0^3 &\text{on $\R$}\\
v_{0}(0)=0\\
v_{0}\to\pm 1 &\text{as $t\to\pm\infty$,}
\end{cases}
$$
with explicit formula $v_0(t):=\tanh(t/\sqrt{2})$. In this way, the nodal set would be exactly the image of the curve 
\begin{eqnarray}\notag
\gamma_{T,\phi}:=\{\gamma_T(s)+\phi(\eps s)\gamma^{'}_T(s)^\bot\}.
\end{eqnarray}
However, this  simple  approximation is not suitable for our purposes, and we need to correct it in two aspects. 
First, in order to recognize the linearized Willmore equation after the Lyapunov-Schmidt reduction, we have 
to improve the accuracy of the solution by adding further correction terms. Secondly, formally expanding in 
$\eps$ the Cahn-Hilliard equation on the above function  we will produce error terms involving derivatives of $\phi$ up to the order six multiplied by high powers of $\eps$; hence we will get an equation that, in principle, we would not be able to solve in $\phi$. In order to avoid this problem, we use a family $\{R_\theta\}_{\theta\geq 1}$ of smoothing operators on periodic functions on $[0,\bar{T}]$, introduced by Alinhac and Gérard (see \cite{AG}), namely operators satisfying
\begin{equation} 
||R_\theta\phi||_{C^{k,\alpha}([0,\bar{T}])}\leq c\, ||\phi||_{C^{k',\alpha '}([0,\bar{T}])}\qquad \text{if }k+\alpha\leq k'+\alpha ';\label{smooth1}
\end{equation}
\begin{equation}
||R_\theta\phi||_{C^{k,\alpha}([0,\bar{T}])}\leq c\, \theta^{k+\alpha-k'-\alpha '}||\phi||_{C^{k',\alpha '}([0,\bar{T}])}\qquad \text{if }k+\alpha\geq k'+\alpha^{'}; \label{smooth2}
\end{equation}
\begin{equation}
||\phi-R_\theta\phi||_{C^{k,\alpha}([0,\bar{T}])}\leq c\, \theta^{k+\alpha-k'-\alpha '}||\phi||_{C^{k',\alpha '}([0,\bar{T}])}\qquad \text{if }k+\alpha\leq k'+\alpha'
\label{smooth3}.
\end{equation}
Such operators are obtained by, roughly, truncating the Fourier modes higher than $\theta$. 
It is possible to find further details in \cite{Dan}, where the periodic  case is specifically treated. This latter issue 
is common to interface constructions, see for example \cite{PR} and \cite{Ri}, and treated in a similar manner.   \\

Now we set $\phi_\star:=R_{1/\eps}\phi$ and we consider the change of variables $t:=z-\phi_\star(\eps s)$, which corresponds to replacing $\phi$ by $\phi_\star$ in the expansion of the Laplacian (\ref{exp_lapl}). The Cahn-Hilliard equation evaluated on the above 
function $v_0(t)$ is formally of order $\eps^2$. To correct the terms of order $\eps^2$ we can consider
\begin{eqnarray}\notag
v_{0}(t)+v_{1,\eps,\phi}(s,t)+v_{2,\eps,\phi}(s,t) 
\end{eqnarray}
as an approximate solution near the curve, as in \cite{Ri}, Subsection \ref{ss:5.1}. Here, $v_{1,\eps,\phi}$ is given by
\begin{equation}\label{eq:v1}
v_{1,\eps,\phi}(s,t):=v_0(t+\phi_\star(\eps s)-\phi(\eps s))-v_0(t),
\end{equation}
and
\begin{equation}\label{eq:v2}
v_{2,\eps,\phi}(s,t):=\varepsilon^{2}(-k^2(\varepsilon s)+\varepsilon L\phi_\star(\varepsilon s))\eta(t)+\eps^2\big(\phi_\star^{'}(\varepsilon s)\big)^2\tilde{\eta}(t),
\end{equation}
where
\begin{equation}\label{eq:L}
L\phi:=-2k\phi^{''}-2k^{3}\phi,
\end{equation}
and
\begin{eqnarray}\notag
\eta(t)=-v_{0}^{'}(t)\int_{0}^{t}(v_{0}^{'}(s))^{-2}ds\int_{-\infty}^{s}\frac{\tau (v_{0}^{'}(\tau))^{2}}{2}d\tau.
\end{eqnarray}
The function $\eta$ is exponentially decaying, odd in $t$ and solves 
\begin{eqnarray}\notag
L_{\star}\eta(t):=-\eta^{''}(t)+W^{''}(v_{0}(t))\eta(t)=\frac{1}{2}tv_{0}^{'}(t)\\\notag
\int_{\R}\eta(t)v_{0}^{'}(t)dt=0.
\end{eqnarray}
Here $L_\star$ represents the second variation of the one-dimensional Allen-Cahn energy evaluated at $v_0$. Similarly, $\tilde{\eta}(t):=-tv_0^{'}(t)/2$ solves
\begin{eqnarray}\notag
L_{\star}\tilde{\eta}(t)=v_{0}^{''}(t)\\\notag
\int_{\R}\tilde{\eta}(t)v_{0}^{'}(t)dt=0.
\end{eqnarray}
We note that, in particular, $L_\star^2 \eta=-v_0^{''}$. 

In order to understand the $x_2$-dependence of $u_T$, see \eqref{eq:approx-monot}, we are interested in 
determining the main-order term of $\phi$, which will turn out to be of order $\eps$. 
We then take $\phi$ of the form
\begin{eqnarray}\notag
\phi:=\eps (h+\psi),
\end{eqnarray}
where $h$ is an explicit multiple of $\overline{\phi}$ (see Propositions \ref{p:inv-rhs} and \ref{prop_bifo}) and $\psi$ is some fixed small $C^{4,\alpha}_{\bar{T}}(\R)$-function, in the sense that $||\psi||_{C^{4,\alpha}(\R)}<c \, \eps$, for some constant $c>1$ to be determined later with the aid of a fixed point argument. Now we have to compute the error, that is we have to apply the Cahn-Hilliard operator $F$ 
\begin{equation}\label{eq:F(u)}
  F (u)  = -\Delta(-\Delta u+W^{'}(u))+W^{''}(u)(-\Delta u+W^{'}(u))
\end{equation}
to the approximate solution. Since the approximate solution is defined in a neighbourhood of the perturbed curve $\gamma_{T,\phi}$, namely in $V_{\eps,\phi}$, all the computations below will be performed in the coordinates $(s,t)\in\R\times(-1/4\eps,1/4\eps)$. 
Using then cut-off functions, we then extend $F$ to be identically zero for $|t|\ge 1/4\eps$.

It turns out that, thanks to this last choice of the approximate solution, the Willmore equation, (\ref{exp_lapl}) and a Taylor 
expansion of the potential $W$ in \eqref{eq:W}, the error is of order $\eps^3$. More precisely, for $|t|<1/4\eps$ we have
\begin{eqnarray}
F(v_0+v_{1,\eps,\phi}+v_{2,\eps,\phi})=-\frac{3}{2}\eps^3 k^3(2tv_0^{''}+v_0^{'})
+\eps^4\{-4k^4 t^2 v_0^{''}+k^4\eta^{''}-tv_0^{'}(3k^4+(k')^2)\}\label{Fv=Oeps3} \\\notag
+\eps^5\{-(h_\star^{(4)}+\psi_\star^{(4)})v_0^{'}+(h_\star^{''}+\psi_\star^{''})(3k^2 tv_0^{''}-k^2 v_0^{'}+6k^2(v_0^{'}+2tv_0^{''}))\\\notag
+(h_\star^{'}+\psi_\star^{'})kk'(8tv_0^{''}-v_0^{'}+6(v_0^{'}+2tv_0^{''}))+(h_\star+\psi_\star)(-9k^4 tv_0^{''}-4k^4 v_0^{'}-3(k')^2v_0^{'})\\\notag
+(h_\star-h+\psi_\star-\psi)^{(4)}v_0^{'}+2k(h_\star^{'}+\psi_\star^{'})^2(2v_0^{'''}-tL_\star v_0^{''})\\\notag
+k^5(-4t^2v_0^{'}-5t^3 v_0^{''}+2t\eta^{''}+12\eta\eta^{'}v_0+6\eta^2 v_0^{'}+\frac{5}{2}\eta^{'})+k(k')^2(-9t^2 v_0^{'}-4\eta^{'})\}+\eps^6 F^1_\eps(\psi). 
\end{eqnarray}
In (\ref{Fv=Oeps3}), the right-hand side is evaluated at $(\eps s,t)$. The term $F^1_\eps$ is defined to be identically zero for $|t|\ge 1/4\eps$ and can be suitably estimated using weighted norms. 
To introduce these, for any $0<\delta<\sqrt{2}$ and $x=(x_{1},x_{2})\in\R^{2}$ define the line integral
\begin{equation}\label{eq:G-delta}
\varphi_{\eps,\delta}(x):=\int_{\gamma_T} G_\delta(x,y)dl(y),
\end{equation}
where $G_\delta$ is the Green function of $-\Delta+\delta^2$ in $\R^2$. Notice that the periodicity of 
$\gamma_{T}$ and the exponential decay of  $G_\delta$ make the above integral converge. We denote by $C^{n,\alpha}(\R^{2})$ the space of functions $u:\R^{2}\to\R$ that are $n$ times differentiable and whose $n$-th derivatives are H\"{o}lder continuous with exponent $\alpha$. 
For $L:=(\gamma_T)_1(T)-(\gamma_T)_1(0)$, we set 
\begin{eqnarray}\notag
C^{n,\alpha}_{L,\delta}(\R^{2}):=\bigg\{u\in C^{n,\alpha}(\R^{2}):||u\varphi_{\eps,\delta}^{-1}||_{C^{n,\alpha}(\R^2)}<\infty,u(x_1,x_2)=-u(-x_1,-x_2)
=-u\bigg(x_1+\frac{L}{2},-x_2\bigg)\bigg\},
\end{eqnarray}
where
\begin{eqnarray}
||u||_{C^{n,\alpha}(\R^{2})}=\sum_{j=0}^{n}||\nabla^{j}u||_{L^{\infty}(\R^{2})}+\sup_{x\neq y}\sup_{|\beta|=n}\frac{|\partial_{\beta}u(x)-\partial_{\beta}u(y)|}{|x-y|^{\alpha}}.
\label{norm_1}
\end{eqnarray}
Functions belonging to these spaces decay exponentially away from the curve $\gamma_T$, with rate $e^{-\delta d(\cdotp,\gamma_T)}$, and satisfy its symmetries, that is they are even, periodic with period $L$, and they change sign after translation of half a period and a reflection about the $x_2$ axis. We endow these spaces with the norms
\begin{equation} \label{eq:Cnad}
||u||_{C^{n,\alpha}_\delta(\R^2)}:=||u\varphi_{\eps,\delta}^{-1}||_{C^{n,\alpha}(\R^2)}.
\end{equation}
Using this notation and recalling \eqref{eq:chil}, we have that the error term $F^1_\eps$ in \eqref{Fv=Oeps3} satisfies 
\begin{eqnarray}
\begin{cases}
\|\chi_4 F^1_\eps(\psi)\|_{C^{0,\alpha}_\delta(\R^2)}\leq c\\
\|\chi_4 F^1_\eps(\psi_1)-\chi_4 F^2_\eps(\psi_2)\|_{C^{0,\alpha}_\delta(\R^2)}\leq c\, \|\psi_1-\psi_2\|_{C^{4,\alpha}(\R)},
\end{cases}
\label{small_remainder}
\end{eqnarray}
for $\psi,\psi_1,\psi_2\in C^{4,\alpha}_{\bar{T}}(\R)$ such that $||\psi||_{C^{4,\alpha}(\R)},||\psi_i||_{C^{4,\alpha}(\R)}<1$ , $i=1,2$.

In what follows, we will solve the Cahn-Hilliard equation through a Lyapunov-Schmidt reduction, and to deal with the 
bifurcation equation  we will need to consider the projection of the error terms in the Cahn-Hilliard equation along the 
{\em kernel} of its linearised operator. Fixing $s$, this corresponds to multiplying the error term by $v_0^{'}(t)$ (and a 
cut-off function in $t$) and integrating in $t$. 
 For instance, if $\chi_l$ is the cut-off function introduced at the beginning of this Subsection, the projection 
\begin{eqnarray}\notag
G^1_\eps(\psi)(s):=\int_\R \chi_4(t) F^1_\eps(\psi)(s,t)v_0^{'}(t)dt
\end{eqnarray}
of $\chi_4 F^1_\eps(\psi)$ fulfils
\begin{eqnarray}
\begin{cases}
||G^1_\eps(\psi)||_{C^{0,\alpha}(\R)}\leq c\\
||G^1_\eps(\psi_1)-G^2_\eps(\psi_2)||_{C^{0,\alpha}(\R)}\leq c||\psi_1-\psi_2||_{C^{4,\alpha}(\R)}.
\end{cases}
\label{small_remainder_pr}
\end{eqnarray}
In other words, apart from the coefficient of order $\eps^6$, we get a remainder  which is uniformly bounded for $\psi$ in the unit ball of $C^{4,\alpha}_{\bar{T}}(\R)$, with Lipschitz dependence.

Setting \begin{equation}\label{eq:cstar}
c_\star:=\int_\R (v_0^{'})^2 dt>0,
\end{equation}
and using the fact that 
\begin{eqnarray}\notag
\int_\R tv_0^{''}v_0^{'} dt=-\frac{1}{2}c_\star,
\end{eqnarray}
we can see that the projection of the linear term in $h+\psi$ (and their derivatives) appearing at order $\eps^5$ is given by
\begin{eqnarray}
\int_\R \{-(h_\star^{(4)}+\psi_\star^{(4)})v_0^{'}+(h_\star^{''}+\psi_\star^{''})(3k^2 tv_0^{''}-k^2 v_0^{'}+6k^2(v_0^{'}+2tv_0^{''}))\label{pr_tildeL0}\\\notag
+(h_\star^{'}+\psi_\star^{'})kk'(8tv_0^{''}-v_0^{'}+6(v_0^{'}+2tv_0^{''}))\\\notag
+(h_\star+\psi_\star)(-9k^4 tv_0^{''}-4k^4 v_0^{'}-3(k')^2v_0^{'})+(h_\star-h+\psi_\star-\psi)^{(4)}v_0^{'}\}v_0^{'}dt=\\\notag
-c_\star (\tilde{L}_0(h_\star+\psi_\star)-(h-h_\star+\psi-\psi_\star)^{(4)})=\\\notag
-c_\star(\tilde{L}_0(h_\star+\psi_\star)+(\tilde{L}_0-\frac{d^4}{ds^4})(h_\star-h+\psi_\star-\psi)), 
\end{eqnarray}
where we recall that (see \eqref{eq:tL0})  
\begin{eqnarray}\notag
\tilde{L}_0\phi:=\phi^{(4)}+\frac{5}{2}(k^2 \phi^{'})'+(3(k')^2-\frac{1}{2}k^4)\phi.
\end{eqnarray}
The terms of order $\eps^3$ in \eqref{Fv=Oeps3} can be eliminated by adding to the approximate solution an 
extra correction of the form 
\begin{eqnarray}\notag
v_{3,\eps,\phi}(s,t):=\frac{3}{2}\eps^3k^3 \eta_1,
\end{eqnarray}
 where $\eta_1$ solves
\begin{eqnarray}
\begin{cases}
L_\star^2 \eta_1=2tv_0^{''}+v_0^{'},\\ 
\int_\R \eta_1 v_0^{'} dt=0.
\end{cases}
\label{def_eta1}
\end{eqnarray}
We point out that (\ref{def_eta1}) is solvable since the first right-hand side 
is orthogonal to $v'_0$,  i.e. 
\begin{eqnarray}\notag
\int_\R (2tv_0^{''}+v_0^{'})v_0^{'} dt=0.
\end{eqnarray}
As it is well-known, see e.g. \cite{MW}, the (decaying) kernel of $L_\star$ is generated by $v'_0$ so 
the existence of $\eta_1$ follows from Fredholm's theory.

As a consequence, using (\ref{exp_lapl}) once again and an expansion similar to \eqref{Fv=Oeps3}, for $|t|<1/4\eps$ we have
\begin{eqnarray}
F(v_0+v_{1,\eps,\phi}+v_{2,\eps,\phi}+v_{3,\eps,\phi})=
\eps^4\{-4k^4 t^2 v_0^{''}+k^4\eta^{''}-tv_0^{'}(3k^4+(k')^2)+3k^4 (L_\star\eta_1)'\}\label{approx2}\\\notag
+\eps^5\{(h-h_\star+\psi_\star-\psi)^{(4)}v_0^{'}-(h^{(4)}+\psi_\star^{(4)})v_0^{'}+(h^{''}+\phi_\star^{''})(3k^2 tv_0^{''}-k^2 v_0^{'}+6k^2(v_0^{'}+2tv_0^{''}))\\\notag
+(h_\star^{'}+\psi_\star^{'})kk'(8tv_0^{''}-v_0^{'}+6(v_0^{'}+2tv_0^{''}))+(h+\psi_\star)(-9k^4 tv_0^{''}-4k^4 v_0^{'}-3(k')^2v_0^{'})\\\notag
+k^5(-4t^2v_0^{'}-5t^3 v_0^{''}+2t\eta^{''}+12\eta\eta^{'}v_0+6\eta^2 v_0^{'}+\frac{5}{2}\eta^{'}+3t(L_\star\eta_1)'-\frac{3}{2}\eta_1^{''}+\frac{9}{2}L_\star\eta_1)\\\notag
+k(k')^2(-9t^2 v_0^{'}-4\eta^{'}-18L_\star \eta_1)+2k(h_\star^{'}+\psi_\star^{'})^2(2v_0^{'''}-tL_\star v_0^{''})\}+\eps^6 \{F^1_\eps(\psi)+F^2_\eps(\psi)\}.
\end{eqnarray}
Here
\begin{eqnarray}\label{eq:Qvw}
Q(v,w):=-(W^{'''}(v_0)vw)^{''}+W^{''}(v_0)W^{'''}(v_0)vw+W^{'''}(v_0)(v L_\star w+w L_\star v)=\\\notag
L_\star(W^{'''}(v_0)vw)+W^{'''}(v_0)(v L_\star w+w L_\star v),
\end{eqnarray}
Notice that $F(v_0+v_{1,\eps,\phi}+v_{2,\eps,\phi}+v_{3,\eps,\phi})$ and $F^2_\eps(\psi)$ vanish identically for $|t|\ge 1/4\eps$, $F^2_\eps(\psi)$ satisfies (\ref{small_remainder}) and the projection
\begin{eqnarray}\notag
G^2_\eps(\psi)(s):=\int_\R \chi_4(t) F^2_\eps(\psi)(s,t)v_0^{'}(t)dt
\end{eqnarray}
fulfils estimates similar to \eqref{small_remainder_pr}. Once again, in (\ref{approx2}) the right-hand side is evaluated at $(\eps s,t)$. Similarly, we can improve our approximate solution by correcting the terms of order $\eps^4$ in \eqref{approx2}. By 
the second equality in \eqref{eq:Qvw} and some integration by parts we have that 
\begin{eqnarray}\notag
\int_\R \{-3(L_\star\eta_1)'+4t^2 v_0^{''}+3tv_0^{'}-\eta^{''}-\frac{1}{2}Q(\eta,\eta)\}v_0^{'} dt=\int_\R t(v_0^{'})^2 dt=0.
\end{eqnarray}
Therefore by the above comments we can solve
\begin{eqnarray}\notag
\begin{cases}
L_\star^2 \eta_{2}=-3(L_\star\eta_1)'+4t^2 v_0^{''}+3tv_0^{'}-\eta^{''}-\frac{1}{2}Q(\eta,\eta);\\
\int_\R \eta_2 v_0^{'} dt=0, 
\end{cases}
\end{eqnarray}
and
\begin{eqnarray}\notag
\begin{cases}
L_\star^2 \eta_3=tv_0^{'};\\
\int_\R \eta_3 v_0^{'} dt=0.
\end{cases}
\end{eqnarray}
Finally, we set
\begin{equation}\label{eq:approx-sol-final}
\tilde{v}_{\eps,\phi}(s,t):=v_0(t)+v_{1,\eps,\phi}(s,t)+v_{2,\eps,\phi}(s,t)+v_{3,\eps,\phi}(s,t)+v_{4,\eps,\phi}(s,t),
\end{equation}
where $v_{4,\eps,\phi}(s,t)=\eps^4(k^4 \eta_2+(k')^2\eta_3)$. For $|t|<1/4\eps$, using expansions similar to the 
previous ones we compute
\begin{eqnarray}
F(\tilde{v}_{\eps,\phi})=\eps^5 \{-(h_\star^{(4)}+\psi_\star^{(4)})v_0^{'}+(h_\star^{''}+\psi_\star^{''})(3k^2 tv_0^{''}-k^2 v_0^{'}+6k^2(v_0^{'}+2tv_0^{''}))\label{Ftildev}\\\notag
+(h_\star^{'}+\psi_\star^{'})kk'(8tv_0^{''}-v_0^{'}+6(v_0^{'}+2tv_0^{''}))\\\notag
+(h_\star+\psi_\star)(-9k^4 tv_0^{''}-4k^4 v_0^{'}-3(k')^2v_0^{'})+(h_\star-h+\phi_\star-\phi)^{(4)}v_0^{'}\\\notag
+E_5(\eps s,t)+2k(h_\star^{'}+\psi_\star^{'})^2(2v_0^{'''}-tL_\star v_0^{''})\}+\eps^6\{F^1_\eps(\psi)+F^2_\eps(\psi)+F^3_\eps(\psi)\}.
\end{eqnarray}
Here
\begin{eqnarray}\notag
E_5(s,t):=k^5(-4t^2 v_0^{'}-5t^3 v_0^{''}+2t\eta^{''}-\frac{3}{2}Q(\eta,\eta_1)+3t(L_\star\eta_1)'
\\\notag
+12\eta\eta^{'} v_0+6\eta^2 v_0^{'}-\frac{3}{2}\eta_1^{''}+\frac{9}{2}L_\star\eta_1+\frac{5}{2}\eta^{'}\\\notag
+2(L_\star\eta_2)')+k(k')^2(2(L_\star\eta_3)'-9t^2 v_0^{'}-4\eta^{'}-18 L_\star\eta_1),
\end{eqnarray}
$F^3_\eps$ is identically zero for $|t|\ge 1/4\eps$ and satisfies the counterpart of (\ref{small_remainder}) and
\begin{eqnarray}\notag
G^3_\eps(\psi)(s):=\int_\R \chi_4(t) F^3_\eps(\psi)(s,t)v_0^{'}(t)dt
\end{eqnarray}
fulfils estimates similar to (\ref{small_remainder_pr}). In order to handle the error, we introduce a suitable function space and we endow it with an appropriate weighted norm. We set, for $0<\delta<\sqrt{2}$,
\begin{eqnarray}
\psi_\delta(x):=\zeta(|x_2|)+(1-\zeta(|x_2|))e^{-\delta|x_2|} 
\label{defpsi}
\end{eqnarray}
where $\zeta$ is defined in (\ref{def_zeta}). Now, we define the spaces
\begin{eqnarray}\notag
D^{n,\alpha}_{T,\delta}(\R^2):=\bigg\{U\in C^{n,\alpha}(\R^2):||U\psi_{\delta}||_{C^{n,\alpha}(\R^2)}<\infty,U(x_1,x_2)=-U(-x_1,-x_2)
=-U\bigg(x_1+\frac{T}{2},-x_2\bigg)\bigg\},
\end{eqnarray}
endowed with the norms
\begin{equation} \label{eq:Dnad}
||u||_{D^{n,\alpha}_\delta(\R^2)}:=||u\psi_{\delta}||_{C^{n,\alpha}(\R^2)}.
\end{equation}
The difference between the space $D^{n,\alpha}_{T,\delta}(\R^2)$  and the 
space
$C^{n,\alpha}_{L,\delta}(\R^2)$  introduced previously are the weight function, which depends just on one variable in the case of $D^{n,\alpha}_{T,\delta}(\R^2)$, and the period.

Recalling \eqref{eq:cstar}, define the constant 
\begin{equation}\label{eq:dstar}
d_\star:=\int_\R t^2(v_0^{'})^2 dt>0
\end{equation}
and recall  the definition of $\overline{\phi}$ in Proposition \ref{p:inv-rhs}. 
From the previous computations, we have the following result.

\begin{proposition}\label{p:app-sol-final}
There exist a constant $\tilde{c}>0$ such that
\begin{eqnarray}
||F(\tilde{v}_{\eps,\phi})||_{D^{0,\alpha}_\delta(\R^2)}\leq \tilde{c}\, \eps^5,\label{error_eps5}
\end{eqnarray}
for any $\phi\in C^{4,\alpha}_{\bar{T}}(\R)$ such that $\phi=\eps(\frac{d_\star}{c_\star}\overline{\phi}+\psi)$, with $||\psi||_{C^{4,\alpha}(\R)}<1$.
\end{proposition}

\

\begin{remark}\label{r:expan}
The estimate in Proposition \ref{p:app-sol-final} holds for any function $\phi$ of order $\eps$ in $C^{4,\alpha}$ norm. However, 
for later purposes, we will need to take $\phi=\eps(\frac{d_\star}{c_\star}\overline{\phi}+\psi)$ as above 
in order to determine the principal term in the expansion after projecting onto $v'_0$, when dealing 
with the bifurcation equation.  
\end{remark}

\section{The Lyapunov-Schmidt reduction}\label{s:red}

Up to now, we have only constructed an approximate solution to \eqref{Cahn-Hilliard}, not a true solution, since $F(v_{\varepsilon,\phi})$ is small but not zero (see \eqref{eq:F(u)} and \eqref{defv}). Therefore we try to add a small correction $w=w_{\varepsilon,\phi}:\R^{2}\to\R$ in such a way that $F(v_{\varepsilon,\phi}+w)=0$. Rephrasing our problem in this way, the unknowns are $\phi$ and $w$, for any $\varepsilon>0$ small but fixed (recall that $\eps = \frac{\bar{T}}{T}$). Expanding $F$ in Taylor series, our equation becomes
\begin{equation*}
F(v_{\varepsilon,\phi})+F^{'}(v_{\varepsilon,\phi})[w]+Q_{\varepsilon,\phi}(w)=0,
\label{taylor}
\end{equation*}
where
\begin{eqnarray}
Q_{\varepsilon,\phi}(w)=\int_{0}^{1}dt\int_{0}^{t}F^{''}(v_{\varepsilon,\phi}+sw)[w,w]ds. 
\label{quadratic_w}
\end{eqnarray}
In order to study \eqref{taylor}, we use a Lyapunov-Schmidt reduction, consisting in an {\em auxiliary equation} in $w$ and 
a {\em bifurcation equation} in $\phi$.

\subsection{The auxiliary equation: a gluing procedure}

Recalling the definition of the cut-off $\chi_{l}$ in Subsection \ref{ss:constr}, 
we look for a correction $w$ of the following form
\begin{eqnarray}\notag
w(x)=\chi_{2}(x)\hat{U}(x)+V(x), 
\end{eqnarray}
where $V,\hat{U}$ are defined in $\R^{2}$.  Since $\hat{U}$ is multiplied by a cut-off function that is identically zero far from $\gamma_T$, we look for some suitable function $U=U(t,s)$ defined in $\R^{2}$, then 
we set (see \eqref{eq:Veps})
\begin{eqnarray}\notag
\hat{U}(x):=
\begin{cases}
U(Z_{\varepsilon}^{-1}(x)), &\text{if $x\in V_{\varepsilon}$}\\
0 &\text{if $x\in\R^{2}\backslash V_{\varepsilon}$}.
\end{cases}
\end{eqnarray}
As above, the value of $\hat{U}$ far from the curve does not  matter, since it is multiplied by a cut-off function.\\

\begin{remark}
Let $\tilde{v}_{\eps, \phi}$ be as in Proposition \ref{p:app-sol-final} and let $v_{\eps,\phi}$ be as in \eqref{defv} (see also the 
subsequent formula). Then the potential
\begin{eqnarray}\notag
\Gamma_{\varepsilon,\phi}(x):=(1-\chi_{1}(x))W^{''}(v_{\varepsilon,\phi})+\chi_{1}(x)W^{''}(1)    \qquad \quad 
(W''(1) = 2)
\end{eqnarray}
is positive and bounded away from $0$ in the whole $\R^{2}$. Precisely, for any $0<\delta<\sqrt{2}$, we have $0<\delta^{2}<\Gamma_{\varepsilon,\phi}(x)<2$ provided $\varepsilon$ is small enough, the estimate being uniform in $\phi$. By construction, $\Gamma_{\varepsilon,\phi}\in C^{4,\alpha}(\R^{2})$, it is periodic of period $L$ (the $x_1$-period of $\gamma_{T}$), and the $L^{\infty}$ norms of the derivatives are bounded uniformly in $\phi$ and in $\varepsilon$.
\label{rem_potential}
\end{remark}


Using the fact that $\chi_{2}\chi_{1}=\chi_{1}$, $\chi_2\chi_4=\chi_2$ (recall \eqref{eq:chil}) and the Taylor expansion (\ref{taylor}), we can see that
the Cahn-Hilliard equation $F(v_{\varepsilon,\phi}+w)=0$ can be rewritten as 
\begin{eqnarray}\notag
F(v_{\eps,\phi})+F^{'}(v_{\eps,\phi})w+Q_{\eps,\phi}(w)\\\notag
=\chi_{2}\bigg\{\chi_4 F(\hat{v}_{\varepsilon,\phi})+F^{'}(\hat{v}_{\varepsilon,\phi})\hat{U}+\chi_{1}Q_{\varepsilon,\phi}(\chi_{2}\hat{U}+V)
+\chi_{1}\text{M}_{\varepsilon,\phi}(V)\bigg\}\\\notag+(-\Delta+\Gamma_{\varepsilon,\phi})^{2}V
+(1-\chi_{2})F(v_{\varepsilon,\phi})+(1-\chi_{1})Q_{\varepsilon,\phi}(\chi_{2}\hat{U}+V)+\text{N}_{\varepsilon,\phi}(\hat{U})+\text{P}_{\varepsilon,\phi}(V),
\end{eqnarray}
where
\begin{eqnarray}
\text{M}_{\varepsilon,\phi}(V):=(W^{''}(\hat{v}_{\varepsilon,\phi})-W^{''}(1))(-\Delta V+\Gamma_{\varepsilon,\phi} V)\label{defM}\\\notag
+(-\Delta+W^{''}(\hat{v}_{\varepsilon,\phi}))\big[(W^{''}(\hat{v}_{\varepsilon,\phi})-W^{''}(1))V\big];\\
\text{N}_{\varepsilon,\phi}(\hat{U}):=-2\langle\nabla\chi_{2},\nabla(-\Delta\hat{U}+W^{''}(\hat{v}_{\varepsilon,\phi})\hat{U})\rangle-\Delta\chi_{2}(-\Delta \hat{U}+W^{''}(\hat{v}_{\varepsilon,\phi})\hat{U})\label{defN}\\\notag
+(-\Delta+W^{''}(\hat{v}_{\varepsilon,\phi}))(-2\langle\nabla\chi_{2},\nabla\hat{U}\rangle-\Delta\chi_{2}\hat{U});\\
\text{P}_{\varepsilon,\phi}(V):=-2\langle\nabla\chi_{1},\nabla((W^{''}(\hat{v}_{\varepsilon,\phi})-W^{''}(1))V)\rangle
-\Delta\chi_{1}(W^{''}(\hat{v}_{\varepsilon,\phi})-W^{''}(1))V\label{defP}\\\notag
+W^{'''}(v_{\varepsilon,\phi})(-\Delta v_{\varepsilon,\phi}+W^{'}(v_{\varepsilon,\phi}))V.
\end{eqnarray}
By the expansion of the Laplacian (\ref{exp_lapl}), we can see that, expressing $F'(\tilde{v}_{\eps,\phi})$ in the $(s,t)$-coordinates, for $|t|<1/4\eps$,
\begin{eqnarray}
F^{'}(\tilde{v}_{\eps,\phi})=\mathcal{L}^2+\text{R}_{\eps,\phi},\label{defR}
\end{eqnarray}
where 
\begin{equation} \label{eq:mcal-L}
\mathcal{L}=-(\partial^2_s+\partial^2_t)+W^{''}(v_{0}(t))
\end{equation}
and $\text{R}_{\eps,\phi}=O(\varepsilon)$, in the sense that (recall \eqref{eq:Dnad})
$$
||\chi_4 \text{R}_{\eps,\phi}U||_{D^{0,\alpha}_\delta(\R^2)}\leq c \, \eps||U||_{D^{4,\alpha}_\delta(\R^2)}.
$$
Once again, we have extended $\text{R}_{\eps,\phi}$ to be identically zero for $|t|\ge 1/4\eps$. Hence we have reduced our problem to finding a solution $(\phi,V,U)$ to the system
\begin{numcases}{}
(-\Delta+\Gamma_{\varepsilon,\phi})^{2}V+(1-\chi_{2})F(v_{\varepsilon,\phi})+(1-\chi_{1})Q_{\varepsilon,\phi}(\chi_{2}\hat{U}+V)\label{eq_aux1}\\\notag
+\text{N}_{\varepsilon,\phi}(\hat{U})+\text{P}_{\varepsilon,\phi}(V)=0; &\text{ in $\R^{2}$}\\
\chi_4 F(\tilde{v}_{\varepsilon,\phi})+\mathcal{L}^2 U+\chi_4 \text{R}_{\eps,\phi}U
+\chi_{1}Q_{\varepsilon,\phi}(\chi_{2}U+V(Z_{\varepsilon}(s,t)))\label{eq_aux2}
+\chi_{1}\text{M}_{\varepsilon,\phi}(V)=0 &\text{for $|t|\leq 1/8\varepsilon+4$}.
\end{numcases}
It is understood that, in equation (\ref{eq_aux2}), the cut-off functions and $V$ are evaluated at $Z_\eps(s,t)$, see 
\eqref{def_Z}. First we fix $\phi$ and $U$ and we solve the auxiliary equation (\ref{eq_aux1}) by a fixed point argument, using the coercivity of the operator $(-\Delta+\Gamma_{\varepsilon,\phi})^{2}$.  This is possible due to fact that the potential $\Gamma_{\varepsilon,\phi}$ is bounded from above and from below by positive constants (see Remark \ref{rem_potential}).\\

We have next the following result, that will be proved in Section \ref{s:tech} (recall \eqref{eq:Dnad} and \eqref{eq:Cnad}).

\begin{proposition}
For any $\varepsilon>0$ small enough, for any $U\in D^{4,\alpha}_{T,\delta}(\R^{2})$ such that $||U||_{D^{4,\alpha}_{\delta}(\R^{2})}< 1$ and for any $\phi\in C^{4,\alpha}_{\bar{T}}(\R)$ with $||\phi||_{C^{4,\alpha}(\R)}<1$, equation (\ref{eq_aux1}) admits a solution $V_{\varepsilon,\phi,U}\in C^{4,\alpha}_{L,\delta}(\R^{2})$ satisfying
\begin{eqnarray}\notag
\begin{cases}
||V_{\varepsilon,\phi,U}||_{C^{4,\alpha}_\delta(\R^{2})}\leq c_{1}e^{-\delta/8\varepsilon};\\
||V_{\varepsilon,\phi,U_{1}}-V_{\varepsilon,\phi,U_{2}}||_{C^{4,\alpha}_\delta(\R^{2})}\leq c_{1}e^{-\delta/8\varepsilon}||U_{1}-U_{2}||_{D^{4,\alpha}_{\delta}(\R^{2})};\\
||V_{\varepsilon,\phi_{1},U}-V_{\varepsilon,\phi_{2},U}||_{C^{4,\alpha}_\delta(\R^{2})}\leq c_{1}e^{-\delta/8\varepsilon}||\phi_{1}-\phi_{2}||_{C^{4,\alpha}(\R)},
\end{cases}
\end{eqnarray}
for any $U_{1},U_{2}$ with $||U_{1}||_{D^{4,\alpha}_{\delta}(\R^{2})},||U_{2}||_{D^{4,\alpha}_{\delta}(\R^{2})}<1$, for any $\phi_{1},\phi_{2}\in C^{4,\alpha}_{\bar{T}}(\R)$ with $||\phi_{i}||_{C^{4,\alpha}(\R)}<1$, $i=1,2$, and for some constant $c_{1}>0$ independent of $U$, $\varepsilon$ and $\phi$.
\label{propaux_1}
\end{proposition}

\

Since we reduced solving the Cahn-Hilliard equation to the system \eqref{eq_aux1}-\eqref{eq_aux2}, 
it remains to solve the second component. The operator $\mathcal{L}^2$ (see \eqref{eq:mcal-L}) is not uniformly 
coercive as $\eps \to 0$: in fact, in the $t$ component it annihilates $v'_0(t)$, while due to the fact that $s$ 
lies in an expanding domain, the spectrum of $\partial_s^2$ approaches zero. 
Due to the consequent lack of invertibility of $\mathcal{L}^2$ we need some orthogonality condition to solve equation $\mathcal{L}^2 U=f$, that is
\begin{eqnarray}\notag
\int_{\R}f(s,t)v_0^{'}(t)dt=0 &\forall s\in\R,
\end{eqnarray}
as we will see in Subsection $5.1$, and the solution will satisfy the same orthogonality condition (for a detailed discussion, see Section $5$). As a consequence, equation (\ref{eq_aux2}) cannot be solved directly, through a fixed point argument, hence we subtract the projection along $v'_0$ of the right-hand side. In other words, setting
\begin{eqnarray}
\text{T}(U,V,\phi):=\chi_{1}Q_{\varepsilon,\phi}(\chi_2 U+V)+\chi_4 \text{R}_{\varepsilon,\phi}(U)+\chi_{1}\text{M}_{\varepsilon,\phi}(V); \label{defT}
\end{eqnarray}
\begin{eqnarray}
p_{\phi}(s):=\frac{1}{c_{\star}}\int_{-\infty}^{\infty}\big(\chi_4 F(\tilde{v}_{\varepsilon,\phi})
+\text{T}(U,V_{\varepsilon,\phi,U},\phi)\big)(s,t)v_{0}^{'}(t)dt\label{defp},
\end{eqnarray}
we can solve 
\begin{eqnarray}
\mathcal{L}^{2}U=-\chi_4 F(\tilde{v}_{\varepsilon,\phi})-\text{T}(U,V_{\varepsilon,\phi,U},\phi)
+p_\phi(s)v_{0}^{'}(t)\label{proj_prob}\\
\int_\R U(s,t)v_{0}^{'}(t)dt=0 &\forall s\in\R. \nonumber
\end{eqnarray}
in $U$, for any small but fixed $\phi\in C^{4,\alpha}_{\bar{T}}(\R)$. Concerning the operator near $\gamma_T$, we have the following result, that will be proved in Section \ref{s:tech}.

\begin{proposition}
For any $\varepsilon>0$ small enough and for any $\phi\in C^{4,\alpha}_{\bar{T}}(\R)$ with $||\phi||_{C^{4,\alpha}(\R)}<1$, we can find a solution $U_{\varepsilon,\phi}\in D^{4,\alpha}_{T,\delta}(\R^2)$ to equation (\ref{proj_prob}) satisfying the orthogonality condition
\begin{eqnarray}
\int_{\R}U_{\eps,\phi}(s,t)v_{0}^{'}(t)dt=0, &\forall s\in\R
\end{eqnarray}
and the estimates
\begin{eqnarray}
\begin{cases}
||U_{\varepsilon,\phi}||_{D^{4,\alpha}_{\delta}(\R^2)}\leq c_{2}\, \varepsilon^{5}\\
||U_{\varepsilon,\phi_{1}}-U_{\varepsilon,\phi_{2}}||_{D^{4,\alpha}_{\delta}(\R^2)}\leq c_{2}\, \varepsilon^{5}||\phi_{1}-\phi_{2}||_{C^{4,\alpha}(\R)},
\end{cases}
\end{eqnarray}
for any $\phi_{1},\phi_{2}\in C^{4,\alpha}_{T}(\R)$ with $||\phi_i||_{C^{4,\alpha}(\R)}<1$, $i=1,2$, for some constant $c_{2}>0$ independent of $\varepsilon$.
\label{propaux_2}
\end{proposition}

\subsection{The bifurcation equation}\label{ss:bif}

Using the notation in the previous subsection (see in particular the discussion before Proposition \ref{propaux_2}), the Cahn-Hilliard equation 
reduces to 
$$
  \mathcal{L}^{2}U=-\chi_4 F(\tilde{v}_{\varepsilon,\phi})-\text{T}(U,V_{\varepsilon,\phi,U},\phi). 
$$
Recalling (\ref{proj_prob}), in order to conclude the proof it remains to solve the bifurcation equation 
\begin{equation}
p_{\phi}(s)=0  \qquad \qquad \hbox{ for all } s\in\R
\label{eq_bifo}
\end{equation}
with respect to $\phi$, where $p_\phi$ is the projection of the right-hand side of equation (\ref{eq_aux2}) along $v'_0$ (see (\ref{defp}) and (\ref{defT})). Since the Cahn-Hilliard functional is related via Gamma convergence to the Willmore's, 
the principal part of the bifurcation equation turns out to be the linearized Willmore's, appearing in the second variation of the Willmore energy.  
Recalling (33) from \cite{LMS},  on a hypersurface $\Sigma$ the latter second variation is given by
\begin{eqnarray}\notag
\mathcal{W}^{''}(\Sigma)[\phi,\psi]=\int_{\Sigma}(\tilde{L}_{0}\phi) \psi \, d\sigma,
\end{eqnarray}
where $d\sigma$ is the area form and $\tilde{L}_{0}$ is the self-adjoint operator given by
\begin{eqnarray}\notag
\tilde{L}_{0}\phi=L_{0}^{2}\phi+\frac{3}{2}H^{2}L_{0}\phi-H(\nabla_{\Sigma}\phi,\nabla_{\Sigma}H)+2(A\nabla_{\Sigma}\phi,\nabla_{\Sigma}H)+\\\notag
2H \langle A,\nabla^{2}\phi \rangle +\phi(2 \langle A,\nabla^{2}H \rangle +|\nabla_{\Sigma}H|^{2}+2H\tr A^{3}).
\end{eqnarray}
Here, $L_0\phi=-\Delta_\Sigma \phi-|A|^2\phi$ is the Jacobi operator (related to the second variation of the area functional), $A$ is the second fundamental form, $H$ is the mean curvature and $\tr A^3$ is the trace of $A^3$. Recalling \eqref{eq:Will} and \eqref{eq:cons}, on planar curves  
%
$\tilde{L}_{0}$ can be written as  
\begin{eqnarray}\notag
\tilde{L}_{0}\phi=\phi^{(4)}+\frac{5}{2}(\phi^{'}k^{2})^{'}+(3-\frac{5}{4}k^{4})\phi=\\\notag
\phi^{(4)}+\frac{5}{2}k^{2}\phi^{''}+5kk^{'}\phi^{'}+(3-\frac{5}{4}k^{4})\phi.
\end{eqnarray}
\begin{lemma}
Recalling the definition of the constants (see \eqref{eq:cstar} and \eqref{eq:dstar}) 
$$
   c_\star:=\int_\R (v_0^{'})^2 dt>0, \qquad \qquad d_\star:=\int_\R t^2(v_0^{'})^2 dt>0, 
$$ 
the bifurcation equation can be written in the form
\begin{eqnarray}\notag
\eps^{-1}c_\star\tilde{L}_0\phi=c_\star\tilde{L}_0(h+\psi)=d_\star\overline{g}+\eps \, \mathcal{G}_\eps(\psi), 
\end{eqnarray}
where $\mathcal{G}_\eps$ satisfies estimates similar to \eqref{small_remainder_pr}.\label{lemma_bifo}
\end{lemma}

\begin{proof}
In view of (\ref{Ftildev}) and (\ref{pr_tildeL0}) (see also Subsection \ref{ss:constr} for the definition of the $G^i_\eps$'s and for the 
Fourier-truncation $\phi \mapsto \phi_\star$) one has 
\begin{eqnarray}\notag
\int_\R  F(\tilde{v}_{\eps,\phi})(s,t)v_0^{'}(t)dt=-\eps^5 c_\star(\tilde{L}_0(h+\psi)+(\tilde{L}_0-\frac{d^4}{ds^4})(h_\star-h+\psi_\star-\psi))\\\notag
+\eps^5 \int_\R E_5(\eps s,t)v_0^{'}(t)dt+2 k \, \eps^5 (h_\star^{'}+\psi_\star^{'})^2+\eps^5\int_\R (2v_0^{'''}-tL_\star v_0^{''})v_0^{'}(t)dt\\\notag
+G^1_\eps(\psi)+G^2_\eps(\psi)+G^3_\eps(\psi)+G^4_\eps(\psi),
\end{eqnarray}
where $G^1_\eps(\psi),G^2_\eps(\psi),G^3_\eps(\psi)$ are defined in subsection \ref{ss:constr} and
$$
G^4_\eps(\psi)(s):=\int_\R (\chi_4(t)-1)F(\tilde{v}_{\eps,\phi})(s,t)v_0^{'}(t) dt
$$
is exponentially small in $\eps$, thus in particular it also satisfies the counterpart of \eqref{small_remainder_pr}. Integrating by parts, it is possible to see that the last term vanishes. By the properties of the smoothing operators (see (\ref{smooth3})), the term of order $\eps^5$ satisfies
\begin{eqnarray}\notag
\left\|(\tilde{L}_0-\frac{d^4}{ds^4})(h_\star-h+\psi_\star-\psi)\right\|_{C^{0,\alpha}(\R)}\leq c \, \eps^2||h+\psi||_{C^{4,\alpha}(\R)},
\end{eqnarray} 
since $\tilde{L}_0-\frac{d^4}{ds^4}$ is a second-order differential operator. It remains to deal with the contribution of the term involving $E_5$. We compute
\begin{eqnarray}\notag
c_\star:=\int_\R (v_0^{'})^2 dt=2\lim_{x\to\infty} \int_0^x (v_0^{'})^2(t) dt =\\\notag
2\lim_{x\to\infty} \frac{\left(3 \sinh \left(\frac{x}{\sqrt{2}}\right)+\sinh \left(\frac{3 x}{\sqrt{2}}\right)\right) \text{sech}^3\left(\frac{x}{\sqrt{2}}\right)}{6 \sqrt{2}}=\frac{2\sqrt{2}}{3},
\end{eqnarray}
and
\begin{eqnarray}\notag
d_\star:=\int_\R t^2(v_0^{'})^2dt=2\lim_{x\to\infty}\int_0^x t^2(v_0^{'})^2dt=\\\notag
2\lim_{x\to\infty}\frac{1}{18} \bigg(12 \sqrt{2} \text{Li}_2(-e^{-\sqrt{2} x})-6 \sqrt{2} x^2+6 \sqrt{2} x^2 \tanh (\frac{x}{\sqrt{2}}) \\\notag
- 24 x \log (e^{-\sqrt{2} x}+1)-6 \sqrt{2} \tanh (\frac{x}{\sqrt{2}}) \\\notag
+3 x (\sqrt{2} x \tanh (\frac{x}{\sqrt{2}})+2) \text{sech}^2(\frac{x}{\sqrt{2}})+\sqrt{2} \pi ^2\bigg)=\frac{\sqrt{2}}{9}(\pi^2-6).
\end{eqnarray}
Here $\text{Li}_2$ is the dilogarithmic function, defined as 
$$
 \text{Li}_2(x) = - \int_1^x \frac{\log t}{t-1} dt, 
$$
and satisfies 
$$
   \frac{d}{d x} \text{Li}_2\left(-e^{-\sqrt{2} x}\right) = \sqrt{2} \log \left(e^{-\sqrt{2} x}+1\right), 
   \qquad \hbox{ with } \quad \text{Li}_2\left(0\right) = 0, \, \text{Li}_2\left(-1\right) = -\frac{\pi ^2}{12}. 
$$
For these and further details about the dilogarithmic function, see for instance \cite{AS}, page $1004$. 

Now we deal with the projection of $E_5$. Integrating by parts, it is possible to see that
\begin{eqnarray}\notag
\int_\R t^3v_0^{''}v_{0}^{'}dt=-\frac{3}{2}d_\star;\\\notag
\int_\R \eta^{'} v_0^{'} dt=\frac{1}{4}d_\star;\\\notag
2\int_\R(L_\star\eta_3)'v_0^{'}dt=-2\int_\R L_\star\eta_3 v_0^{''}dt=\int_\R L_\star\eta_3 L_\star(tv_0^{'})dt;\\\notag
=\int_\R L_\star^2 \eta_3 tv_0^{'} dt=d_\star, 
\end{eqnarray}
and that 
\begin{eqnarray}\notag
k(k')^2\int_\R(2t(L_\star\eta_1)'-9t^2 v_0^{'}-4\eta^{'}-18 L_\star\eta_1)v_0^{'} dt=-9 k(k')^2 d_\star.
\end{eqnarray}
Moreover, since due to \eqref{eq:Qvw}
\begin{eqnarray}\notag
\int_\R Q(\eta,\eta)tv_0^{'} dt=-12\int_\R v_0 v_0^{''}\eta^2 dt +6\int_\R t^2(v_0^{'})^2 v_0 \eta dt,
\end{eqnarray}
we have
\begin{eqnarray}\notag
2\int_\R(L_\star\eta_2)'v_0^{'}=
\int_\R L_\star^2 \eta_2 tv_0^{'} dt=\\\notag
-\frac{3}{2}d_\star-\int_\R t\eta^{''}v_0^{'}dt+6\int_\R v_0 v_0^{''}\eta^2 dt
-3\int_\R t^2 (v_0^{'})^2\eta v_0 dt.
\end{eqnarray}
The quadratic term in $\eta$ gives
\begin{eqnarray}\notag
12\int_\R \eta\eta^{'} v_0^{'}v_0 dt+6\int_\R\eta^2 (v_0^{'})^2 dt=-6\int_\R v_0 v_0^{''} \eta^2 dt.
\end{eqnarray}
With a similar reasoning, the quadratic term containing $\eta$ and $\eta_1$ gives
\begin{eqnarray}\notag
\int_\R Q(\eta,\eta_1)v_0^{'} dt=-3\int_\R v_0 (v_0^{'})^2\eta t^2 dt+\frac{3d_\star}{c_\star}\int_\R v_0 (v_0^{'})^2\eta dt
+3\int_\R tv_0 (v_0^{'})^2\eta_1 dt.
\end{eqnarray}
Moreover,
\begin{eqnarray}\notag
\int_\R t(L_\star\eta_1)'v_0^{'} dt=-\frac{d_\star}{2}.
\end{eqnarray}
We note that
\begin{eqnarray}\notag
L_\star(tv_0 v_0^{'}/\sqrt{2})=v_0^{'''}+3tv_0(v_0^{'})^2,\\\notag
L_\star(tv_0^{'}(1+\sqrt{2}tv_0)/4)=tv_0^{'''}+\frac{3}{2}t^2 v_0(v_0^{'})^2,\\\notag
L_\star(v_0 v_0^{'}/3\sqrt{2})=v_0(v_0^{'})^2.
\end{eqnarray}
thus
\begin{eqnarray}
\int_\R \eta_1(v_0^{'''}+3tv_0(v_0^{'})^2) dt=\frac{1}{\sqrt{2}}\int_\R tv_0 v_0^{'}L_\star\eta_1 dt=-\frac{d_\star}{4},\label{int1}\\
\int_\R \eta (tv_0^{'''}+\frac{3}{2}t^2 v_0(v_0^{'})^2)dt=\frac{1}{4}\int_\R tv_0^{'}(1+\sqrt{2}tv_0)L_\star\eta dt=\frac{5}{16}d_\star,\label{int2}\\
\int_\R \eta v_0(v_0^{'})^2 dt=\frac{1}{3\sqrt{2}}\int_\R v_0 v_0^{'} L_\star\eta dt=\frac{1}{18\sqrt{2}}\label{int3}.
\end{eqnarray}
In order to prove (\ref{int1}), (\ref{int2}) and (\ref{int3}) we observe that, concerning the first integral 
\begin{eqnarray} \nonumber
    \frac{1}{\sqrt{2}}\int_0^x  t v_0 v_0^{'}L_\star\eta_1 dt &=&  \frac{1}{288} \left(-72 \sqrt{2} \text{Li}_2\left(-e^{-\sqrt{2} x}\right)-6 \sqrt{2} \left(\pi ^2-6 x^2\right) \right. \\ \nonumber &+& \left. 2 \sqrt{2} \left(-18 x^2+\pi ^2+12\right) \tanh \left(\frac{x}{\sqrt{2}}\right) \right. 
    \\ &+& \left. 3 x \left(6 x^2-\pi ^2+6\right) \text{sech}^4\left(\frac{x}{\sqrt{2}}\right) \right. 
    \\ &+& \left. \left(\sqrt{2} \left(-18 x^2+\pi ^2-6\right) \tanh \left(\frac{x}{\sqrt{2}}\right)-36 x\right) \nonumber \text{sech}^2\left(\frac{x}{\sqrt{2}}\right) \right. \\ &+& \left. 144 x \log \left(e^{-\sqrt{2} x}+1\right)\right). \nonumber
\end{eqnarray}
Concerning the second integral, one has indeed 
\begin{eqnarray}\notag
\frac{1}{4} \int_0^x t v_0^{'}(1+\sqrt{2}tv_0) L_\star \eta \, dt & = & \nonumber
\frac{1}{288} \left(60 \sqrt{2} \text{Li}_2\left(-e^{-\sqrt{2} x}\right)-9 x^3 \text{sech}^4\left(\frac{x}{\sqrt{2}}\right)\right. 
\\ \nonumber & +& \left. 5 \left(\sqrt{2} \left(\pi ^2-6 x^2\right)+6 \sqrt{2} \left(x^2-1\right) \tanh \left(\frac{x}{\sqrt{2}}\right)\right. \right. 
\\ &-& \left.\left.  24 x \log \left(e^{-\sqrt{2} x}+1\right)\right) \right.  \\ & +& \left.  15 x \left(\sqrt{2} x \tanh \left(\frac{x}{\sqrt{2}}\right)+2\right) \text{sech}^2\left(\frac{x}{\sqrt{2}}\right)\right). \nonumber 
\end{eqnarray}
For the third integral, one has that 
$$
  \frac{1}{3\sqrt{2}}\int_0^x v_0 v_0^{'} L_\star\eta \, dt = \frac{\left(-6 \sqrt{2} x+4 \sinh \left(\sqrt{2} x\right)+\sinh \left(2 \sqrt{2} x\right)\right) \text{sech}^4\left(\frac{x}{\sqrt{2}}\right)}{288 \sqrt{2}}. 
$$

Taking the sum, we get
\begin{eqnarray}
\int_\R E_5(s,t)v_0^{'}(t)dt=\frac{9}{8}d_\star k^5(s)-9k(k')^2(s) d_\star=d_\star \overline{g}.
\end{eqnarray}

To conclude the proof we observe that, thanks to Propositions \ref{propaux_1} and \ref{propaux_2}, 
$$
G^5_\eps(\psi)(s):=\eps^{-6}\int_\R \text{T}(U_\phi,V_{\eps,\phi,U},\phi)(\eps^{-1}s,t)v_0^{'}(t) dt
$$
satisfies (\ref{small_remainder_pr}). Therefore, if we set 
\begin{eqnarray}\label{eq:calGeps}
\mathcal{G}_\eps(\psi):=\eps^{-2}(\frac{d^4}{ds^4}-\tilde{L}_0)(h_\star-h+\psi_\star-\psi)+G^1_\eps(\psi)+G^2_\eps(\psi)+G^3_\eps(\psi)
+G^4_\eps(\psi)+G^5_\eps(\psi), 
\end{eqnarray}
the Lemma is then proved. 
\end{proof}

\

\begin{proposition}
For $\eps>0$ small enough, the bifurcation equation (\ref{eq_bifo}) admits a solution $\phi\in C^{4,\alpha}_{\bar{T}}(\R)$, such that
\begin{eqnarray}\notag
\phi=\eps\frac{d_\star}{c_\star}\overline{\phi}+\eps \, \psi,
\end{eqnarray}
with $\psi\in C^{4,\alpha}_{\bar{T}}$ fulfilling $||\psi||_{C^{4,\alpha}(\R)}\leq c \, \eps$ for some constant $c>0$ 
and where $c_\star$, $d_\star$ are given by \eqref{eq:cstar} and \eqref{eq:dstar}.\label{prop_bifo}
\end{proposition}
\begin{proof}
We recall that we look for a solution of the form $\phi=\eps(h+\psi)$ and, by Lemma \ref{lemma_bifo}, the bifurcation equation can be written in the form
\begin{eqnarray}\notag
\tilde{L}_0(h+\psi)=\overline{g}+\eps \, \mathcal{G}_\eps(\psi),
\end{eqnarray}
where $\mathcal{G}_\eps$ is given by \eqref{eq:calGeps}. 
In order to solve it, first we set $h:=\frac{d_\star}{c_\star}\overline{\phi}$, in such a way that
\begin{eqnarray}\notag
\tilde{L}_0 h=\frac{d_\star}{c_\star}\overline{g},
\end{eqnarray}
(see Proposition \ref{p:inv-rhs}), then we treat the fixed point problem
\begin{eqnarray}\notag
\tilde{L}_0\psi=\eps\, \mathcal{G}_\eps(\psi),
\end{eqnarray}
using the inverse of $\tilde{L}_0$ constructed in Proposition \ref{prop_lin_bifo}. In order to apply the contraction mapping theorem, we need to prove the Lipschitz character of $\mathcal{G}_\eps$. This follows from the definitions of $p_\phi$ and T (see (\ref{defp}) and (\ref{defT})), the Lipschitz regularity of $G^i_\eps$, $i=1,\dots,5$ (they all meet \eqref{small_remainder_pr}), which follows from property (\ref{small_remainder}), satisfied by $F^1_\eps, \, F^2_\eps$ and $F^3_\eps$ and the Lipschitz dependence of $U$ and $V$ on the datum $\phi$ (see Propositions \ref{propaux_2} and \ref{propaux_1}).
\end{proof}

\subsection{Proof of Theorem \ref{t:ex}} 

Thanks to the results in the previous subsections, proving existence of a symmetric solution to \eqref{CH_mult}, 
we only need to prove \eqref{eq:approx-monot}. By the symmetries of $u_T$, we can reduce ourselves to study the sign of $\partial_{x_2}u$ in the strip $\{0\leq x_1\leq (\gamma_T)_1(T/4)\}$.

Before proceeding, similarly to \eqref{def_Z}, for any $x\in V_\eps$ (see \eqref{eq:Veps}) we set 
\begin{eqnarray}\notag
x=\tilde{Z}_\eps(s,z)=\gamma_T(s)+z \, \gamma_T^{'}(s)^\bot; \qquad \qquad \gamma_T(s) = \frac{1}{\eps} \gamma(\eps s). 
\end{eqnarray}
Since $\gamma_T(s) = \gamma'(\eps \, s)$, the latter formula becomes 
$$
  \tilde{Z}_\eps(s,z)= \frac{1}{\eps} \gamma(\eps s)+z \, \gamma^{'}(\eps \, s)^\bot.
$$
We would like to understand the inverse function, namely the dependence of $(t,z)$ on $(x_1, x_2)$, especially near the 
$x_1$-axis. We notice first that $\tilde{Z}_\eps(0,z) = (z,0)$, and that 
$$
  \left( \begin{matrix}
  \partial_{s} x_1 & \partial_{s} x_2 \\ 
  \partial_{z} x_1 & \partial_{z} x_2
  \end{matrix} \right) = \left( \begin{matrix}
    - (1 - \eps \, z \, k(\eps s)) \sin \int_0^{\eps s} k(\tau) d \tau  &   (1 - \eps \, z \, k(\eps s)) \cos  \int_0^{\eps s} k(\tau) d \tau \\ 
    \cos  \int_0^{\eps s} k(\tau) d \tau & \sin \int_0^{\eps s} k(\tau) d \tau 
    \end{matrix} \right). 
$$
Recalling that $k(0) = 0$ and  $k'(0) < 0$, differentiating the definition of $\tilde{Z}_\eps$ and taking the scalar product with $\gamma'(\eps s)^\bot$, it is easy to see that $\partial_{x_2}z=\gamma'_1(\eps s)$ in $V_\eps$, then near the origin one has, 
for $\overline{\delta} > 0$ small 
\begin{equation}\label{eq:dx2z}
  \partial_{x_2} z = \frac{|k'(0)|}{2} x_2^2 \eps^2 (1 + o_\eps(1)); \qquad \qquad x \in V_\eps, \quad |x_2| < \frac{\overline{\delta}}{\eps}. 
\end{equation}
%
%
%
%
%
After these preliminaries, we have  the following result. 

\begin{proposition}\label{p:approx-monot}
Let $\tilde{v}_{\eps,\phi}(s,t)$ be the approximate solution defined in \eqref{eq:approx-sol-final}. 
Then there exists a fixed constant $C$ such that 
$$
  \frac{\partial \tilde{v}_{\eps,\phi}}{\partial x_2} \geq - C \eps^3 \qquad \quad \hbox{ in } V_\eps. 
$$
\end{proposition}

\begin{proof}
Recall that in $V_\eps$ we defined 
$$
  \tilde{v}_{\eps,\phi}(s,t):=v_0(t)+v_{1,\eps,\phi}(s,t)+v_{2,\eps,\phi}(s,t)+v_{3,\eps,\phi}(s,t)+v_{4,\eps,\phi}(s,t). 
$$
We begin by estimating the $x_2$-derivative of the first term. Recalling that $t = z - \phi_\star(\eps s)$, we have 
$$
  \partial_{x_2}  v_0(t) = v'_0(t) \left[ \partial_{x_2} z - \eps \phi'_\star(\eps s) \partial_{x_2} s \right]. 
$$
Concerning the function $\phi'_\star$ we recall that by \eqref{smooth3}, for $\alpha \in (0,1)$ 
and $\theta = \frac{1}{\eps}$ one has 
$$
   \| \phi - \phi_\star \|_{C^{1,\alpha}} \leq C \eps^3 \|\phi\|_{C^{4,\alpha}} \leq C \eps^4. 
$$ 
Moreover, by Proposition \ref{prop_bifo} we had that 
$$
  \left\| \phi - \eps\frac{d_\star}{c_\star} \overline{\phi} \right\|_{C^{4,\alpha}} \leq C \eps^2. 
$$
The latter two formulas imply that near the origin 
$$
  \eps \phi'_\star(\eps s) \partial_{x_2} s \geq - C \eps^3, 
$$
and therefore that also near the $x_1$ axis, by \eqref{eq:dx2z}  
$$
  \partial_{x_2}  v_0(t) \geq - C \eps^3 +\frac{1}{2} x_2^2 \eps^2 (1 + o_\eps(1)).  
$$
Concerning instead $v_{1,\eps,\phi}$, defined in \eqref{eq:v1}, we have that 
$$
  \partial_{x_2} v_{1,\eps,\phi} = (v'_0(t + \phi_\star(\eps s) - \phi(\eps s)) - v'(t)) 
  \left[ \partial_{x_2} z - \eps \phi'_\star(\eps s) \partial_{x_2} s \right] 
  + \eps v'_0(t + \phi_\star(\eps s) - \phi(\eps s)) (\phi'_\star(\eps s) - \phi'(\eps s))\partial_{x_2}s.  
$$
This term can be estimated by 
$$
 C \| \phi_\star - \phi \|_{L^\infty} | \partial_{x_2} z - \eps \phi'_\star(\eps s) \partial_{x_2} s | 
 + C \eps \| \phi_\star - \phi \|_{C^{1,\alpha}}. 
$$
Using \eqref{smooth3} we can check that this term is of order $\eps^5$. 

We turn next to $v_{2,\eps,\phi}$, see \eqref{eq:v2}. The first summand in its definition 
is quite easy to treat. The terms $\eps^3 L \phi_\star (\eps s) \eta(t)$ and 
$\eps^2 \phi'_\star(\eps s)^2 \tilde{\eta}(t)$, involving the Fourier truncation $\phi_\star$ 
might seem more delicate. However, being $L$ of second order (see \eqref{eq:L}), using \eqref{smooth1} 
and recalling that $\|\phi\|_{C^{4,\alpha}} \leq c\, \eps$, one has that the $x_2$-derivative of both these terms 
is of order $\eps^4$. 

All other terms in $\tilde{v}_{\eps,\phi}$ can be estimated  easily, and it is also straightforward to show the monotonicity of 
$\tilde{v}_{\eps,\phi}$ in $x_2$ in $V_\eps$ for $|x_2| \geq \frac{\overline{\delta}}{\eps}$, since here 
$v_0(t)$ has $x_2$-derivative bounded away from zero. 
\end{proof}

\

\begin{pfn} {\em of Theorem \ref{t:ex} completed.}  
We notice that the solution $u_T$ is obtained by multiplying $\tilde{v}_{\eps,\phi}$ by a cut-off function (not 
identically equal to 1 in a region where $\tilde{v}_{\eps,\phi}$ is exponentially small in $\eps$)
and by adding a correction $w$ which is of order $\eps^5$ in $C^1$ norm, see the beginning of Section \ref{s:red}. Then \eqref{eq:approx-monot} follows from Proposition \ref{p:approx-monot}. 

The weighted norm estimate on the correction $w$, the fact that $v'_0(t)$ has non zero gradient for $t$ 
close to zero, and the fact that the {\em tilting } $\phi$ is if order $\eps$ (see Proposition \ref{prop_bifo}) also imply \eqref{eq:dist} by a direct application 
of the implicit function theorem. 
\end{pfn}

\section{Proof of some technical results}\label{s:tech}

Here we collect the proofs of some technical results, most notably of Propositions \ref{propaux_1} and 
 \ref{propaux_2}.

\subsection{Proof of Proposition \ref{propaux_1}}\label{ss:5.1}

Our main strategy is the following: if $\Gamma_{\eps,\phi}$ is as in Remark \ref{rem_potential}, we first we study the linear equation
\begin{eqnarray}
(-\Delta+\Gamma_{\varepsilon,\phi})^{2}u=f&\text{in }\R^{2},
\label{lin_pb_f}
\end{eqnarray}
where $f$ is a fixed function with finite $C^{0,\alpha}_{L,\delta}(\R^{2})$ norm (see \eqref{eq:Cnad}), that is decaying away from the curve $\gamma_T$ 
at rate $e^{-\delta d(\cdotp,\gamma_T)}$, even and periodic with period $L = L_T$ in $x_{1}$ (see Section $4,2$). The aim is to construct a right inverse of the operator $(-\Delta+\Gamma_{\varepsilon,\phi})^{2}$, in order to solve equation (\ref{eq_aux1}) by a fixed point argument (see Subsection \ref{ss:bif}). 
In order to treat equation (\ref{eq_aux1}), we will endow the space $C^{n,\alpha}_{L,\delta}(\R^{2})$ with the norm introduced in (\ref{norm_1}).

\subsubsection{The linear problem}

In order to solve (\ref{lin_pb_f}), we first consider the second order equation
\begin{eqnarray}
-\Delta u+\Gamma_{\varepsilon,\phi} u=f&\text{in }\R^{2},
\label{lin_pb_f_2}
\end{eqnarray}
proving the following result (recall \eqref{def_per_spaces}). 
\begin{lemma}
  Let $f\in C^{0,\alpha}_{L,\delta}(\R^{2})$, $0<\delta<\sqrt{2}$. Then, for $\varepsilon$ small enough and
  $\phi\in C^{4,\alpha}_{T}(\R)$ with $||\phi||_{C^{4,\alpha}(\R)}<1$, equation (\ref{lin_pb_f_2}) admits a unique solution
  $u:=\tilde{\Psi}_{\varepsilon,\phi}(f)\in C^{2,\alpha}_{L,\delta}(\R^{2})$ satisfying
  $||u||_{C^{2,\alpha}_\delta(\R^2)}\leq c \, ||f||_{C^{0,\alpha}_\delta(\R^2)}$, for some constant $c>0$ independent of
  $\varepsilon$ and $\phi$. 
\label{prop_linpb_far}
\end{lemma}
\begin{proof}
\textit{Step (i): existence, uniqueness and local H$\ddot{o}$lder regularity on a strip.}\\

Recalling that $L$ is the $x_1$-period of $\gamma_{T}$, define the strip 
$$
 S = (-L/2,L/2)\times\R. 
$$
First we look for a solution to the Neumann problem
\begin{eqnarray}
\begin{cases}
-\Delta u+\Gamma_{\varepsilon,\phi} u=f&\text{in }S; \\
\partial_{x_{1}}u=0 &\text{on }\partial S,
\label{neumann_pb}
\end{cases}
\end{eqnarray}
and then we will extend it by periodicity to the whole $\R^{2}$. By definition, $w\in H^{1}(S)$ is a weak solution to problem (\ref{neumann_pb}) if 
\begin{eqnarray}
\int_{S} \langle \nabla w,\nabla v \rangle dx+\int_{S}\Gamma_{\varepsilon,\phi}wvdx=\int_{S}fvdx, &\text{for any }v\in H^{1}(S).
\label{weak_sol}
\end{eqnarray}
Existence and uniqueness of such a solution follow from the Riesz representation theorem (see also Remark \ref{rem_potential}). 
Since $f\in C^{0,\alpha}(\overline{S})$, it follows that $w\in C^{2,\alpha}(\overline{S})$. 
Moreover, choosing an arbitrary test function $v\in C^{1}(\overline{S})$ and applying the divergence theorem, we can see that 
\begin{eqnarray}\notag
\int_{S}v(-\Delta w+\Gamma_{\varepsilon,\phi}w)dx+\int_{\partial S}\partial_{x_{1}}wv\,dx_{1}=\int_{S}fv \, dx.
\end{eqnarray}
Taking $v\in C^{\infty}_{0}(S)$, the PDE is satisfied in the classical sense. Taking once
again $v\in C^{1}(\overline{S})$, we have that also $\partial_{x_{1}}w=0$ on $\partial S$. \\

\textit{Step (ii): Symmetry and extension to an entire solution}\\

By the symmetries of the Laplacian and the uniqueness of the solution, if $f$ is even in $x_{1}$ then the same is true for $w$, thus $w(-L/2,x_{2})=w(L/2,x_{2})$ and $\partial^{2}_{x_{1}}w(-L/2,x_{2})=\partial^{2}_{x_{1}}w(L/2,x_{2})$. As a consequence, it is possible to extend $w$ by periodicity to an entire solution $u\in C^{2,\alpha}(\R^{2})$.\\

\textit{Step (iii): $u\in L^{\infty}(\R^{2})$.}\\

By elliptic estimates and the Sobolev embeddings 
\begin{eqnarray}\notag
||u||_{L^{\infty}(B_{1}(x))}\leq c||u||_{W^{2,2}(B_{1}(x))}\leq c\,(||u||_{L^{2}(B_{2}(x))}+||f||_{L^{2}(B_{2}(x))})\\\notag
\leq c\,(||w||_{L^{2}(S)}+||f||_{L^{\infty}(\R^{2})})<\infty
\end{eqnarray}
for any $x\in\R^{2}$, thus $u\in L^{\infty}(\R^{2})$. \\



\textit{Step (iv): Decay of the solution: $u\varphi_{\eps,\delta}^{-1}\in L^{\infty}(\R^{2})$ (see \eqref{eq:G-delta} for 
the definition of $\varphi_{\eps,\delta}$), $0<\delta<\sqrt{2}$.}\\

For suitable constants $\lambda, \tau > 0$, we will use the function $\lambda\varphi+\tau\varphi^{-1}$ as a barrier, where we have set $\varphi:=\varphi_{\eps,\delta}$. More precisely, we fix $\rho>0$ and $z\in\R^{2}$ with $d(z,\gamma_T)>\rho$. Then we fix $\tau>0$ small and $R>|d(z,\gamma_T)|$. Therefore $u$ fulfils
\begin{eqnarray}\notag
u(x)<||u||_{L^{\infty}(\R^{2})}<\lambda\varphi(x)<\lambda\varphi(x)+\tau\varphi^{-1}(x)
\end{eqnarray}
if $d(x,\gamma_T)=\rho$, provided $\lambda>||u||_{L^{\infty}(\R^{2})}\sup_{d(x,\gamma_\eps)=\rho}\varphi^{-1}>0$. Furthermore
\begin{eqnarray}\notag
u<||u||_{L^{\infty}(\R^{2})}<\tau\varphi^{-1}<\lambda\varphi+\tau\varphi^{-1}
\end{eqnarray}
if $d(x,\gamma_T)=R$, provided $R$ is large enough. Moreover,
\begin{eqnarray}\notag
(-\Delta+\Gamma_{\varepsilon,\phi})(u-(\lambda\varphi+\tau\varphi^{-1}))\leq\\\notag (c-\lambda(\Gamma_{\varepsilon,\phi}-\delta^2))\varphi\\\notag
-\tau\varphi^{-1}\left\{\Gamma_{\varepsilon,\phi}+\delta^2-2\frac{|\nabla\varphi|^2}{\varphi^2}\right\}<0 &\text{for $x\in\Omega$,}
\end{eqnarray}
where $\Omega:=\{x:\rho<d(x,\gamma_T)<R\}$, if $\lambda$ is large enough. We observe that, if we fix $0<\beta<W^{''}(1)-\delta^2$, then, for  $\varepsilon$ small enough, 
\begin{eqnarray}
\Gamma_{\varepsilon,\phi}-\delta^2=(\chi_{1}-1)(W^{''}(1)-W^{''}(\tilde{v}_{\varepsilon,\phi}))+W^{''}(1)-\delta^2>W^{''}(1)-\delta^2-\beta>0.
\label{Gamma-delta2>0}
\end{eqnarray}
Thus the function $c/(\Gamma_{\varepsilon,\phi}-\delta^2)$ is bounded from above, therefore we can take $\lambda>\sup_{x\in \Omega}c/(\Gamma_{\varepsilon,\phi}-\delta^2)$. The term multiplying $-\tau\varphi^{-1}$ is positive, due to the estimate $|\nabla\varphi|^2/\varphi^2\leq \delta^2$ and (\ref{Gamma-delta2>0}). Therefore, by the maximum principle we get that $u(z)<\lambda\varphi+\tau \varphi^{-1}$, in the complement of the region $\{|t|\leq\rho\}$ and for any $\tau>0$. In the same way, one can prove that $u(z)>-\lambda\varphi-\tau\varphi^{-1}$. Letting $\sigma\to 0$, we get that $u\varphi^{-1}\in L^{\infty}(\R^{2})$.\\ 

\textit{Step (v): estimate of the $L^{\infty}$-norm of $u \, \varphi_{\eps,\delta}^{-1}$.} \\

Let us set $\tilde{u}:=u\varphi_{\eps,\delta}^{-1}$ and $\tilde f:=f\varphi_{\eps,\delta}^{-1}$. It is possible to possible to show that
\begin{eqnarray}
(-\Delta+\Gamma_{\eps,\phi})\tilde{u}=\tilde{f}-2(\nabla u,\nabla\varphi^{-1})-u\Delta\varphi^{-1}=\label{eq_tildeu}\\\notag
\tilde{f}-2\varphi(\nabla\tilde{u},\nabla\varphi^{-1})+\tilde{u}\left(2\frac{|\nabla\varphi|^2}{\varphi^2}-\varphi\Delta\varphi^{-1}\right)=
\tilde{f}-2\varphi(\nabla\tilde{u},\nabla\varphi^{-1})+\delta^2\tilde{u}, 
\end{eqnarray}
(once again, we have set $\varphi:=\varphi_{\eps,\delta}$ in the above computation). First we assume that there exists a point $y\in\R^{2}$ such that $|\tilde{u}(y)|=\max_{x\in\R^{2}}|\tilde{u}(x)|$. If $\tilde{u}(y)>0$, then $y$ is a maximum point, thus $\nabla\tilde{u}(y)=0$ and
\begin{eqnarray}\notag
(\Gamma_{\varepsilon,\phi}(y)-\delta^2)\tilde{u}(y)\\\notag
\leq-\Delta\tilde{u}(y)+(\Gamma_{\varepsilon,\phi}(y)-\delta^2)\tilde{u}(y)=\tilde{f}(y),
\end{eqnarray}
and therefore
\begin{eqnarray}\notag
||\tilde{u}||_{L^{\infty}(\R^{2})}\leq c\,||\tilde{f}||_{L^{\infty}(\R^{2})}.
\end{eqnarray}
A similar argument shows that the same estimate is true if $\tilde{u}(y)<0$ (a minimum point).\\

If the maximum is not achieved, then there exists a sequence $(x_k)_k \subset\R^2$ such that $|\tilde{u}(x_k)|\to\sup_{x\in\R^2}|\tilde{u}(x)|$.
Since we have periodicity in the $x_1$-variable, we can assume that the $x_2$-component of $x_k$ tends to infinity in 
absolute value.  Then we define $\tilde{u}_k(x):=\tilde{u}(x+x_k)$. Up to a subsequence, $\tilde{u}_k\to w$ in $C^2_{loc}(\R^2)$, 
$\tilde{f}(\cdot + x_k) \to \hat{f}$ in $C^2_{loc}(\R^2)$ and, recalling \eqref{eq:G-delta}, it can be shown that 
$\varphi \simeq e^{- \delta |x_2|}$ for large $|x_2|$ and hence 
$\varphi(\cdot + x_k) 
\nabla \varphi(\cdot + x_k)^{-1} \to \mp \delta \, {\bf e_2}$  if $(x_k)_2 \to \pm \infty$. The limit $w$ solves 
\begin{eqnarray}\notag
-\Delta w+ W''(1) w =\hat{f} \mp 2 \delta \partial_{x_2} w + \delta^2 w  &\text{in $\R^2$.}
\end{eqnarray}
Moreover, $|w(0)|=||w||_{L^\infty(\R^2)}\leq||\tilde{f}||_{L^\infty(\R^2)}$. As a consequence, since $W''(1) = 2$ and $\delta^2 < 2$, we have 
\begin{eqnarray}\notag
||\tilde{u}||_{L^\infty(\R^2)}=||w||_{L^\infty(\R^2)}\leq||\tilde{f}||_{L^\infty(\R^2)}.
\end{eqnarray}

\

\textit{Step (vi): Decay of the derivatives.}\\

By (\ref{eq_tildeu}), step \textit{(v)} and \cite{GT} (Chapter $6.1$, Corollary $6.3$),
\begin{eqnarray}\notag
||\tilde{u}||_{C^{2,\alpha}(B_{1}(x))}\leq c\,(||\tilde{u}||_{L^{\infty}(\R^{2})}+||\tilde{f}||_{C^{0,\alpha}(\R^{2})})\leq c\,||\tilde{f}||_{C^{0,\alpha}(\R^{2})}<\infty,
\end{eqnarray}
for any $x\in\R^2$, thus $u\in C^{2,\alpha}_{\delta}(\R^{2})$ and
\begin{eqnarray}\notag
||\tilde{u}||_{C^{2,\alpha}(\R^2)}\leq c\,||\tilde{f}||_{C^{0,\alpha}(\R^2)}.
\end{eqnarray}
This concludes the proof.
\end{proof}

\

\begin{lemma}
Let $f\in C^{0,\alpha}_{L,\delta}(\R^{2})$, with $0<\delta<\sqrt{2}$. Then, for $\varepsilon$ small enough and $\phi\in C^{4,\alpha}_{T}(\R)$ with $||\phi||_{C^{4,\alpha}(\R)}<1$, equation (\ref{lin_pb_f}) admits a unique solution $V:=\Psi_{\varepsilon,\phi}(f)\in C^{4,\alpha}_{L,\delta}(\R^{2})$ satisfying the estimate $||V||_{C^{4,\alpha}_\delta(\R^{2})}\leq c\,||f||_{C^{0,\alpha}_\delta(\R^{2})}$ for some constant $c>0$ independent of $\varepsilon$ and $\phi$.
\label{prop_linpb_far_4}
\end{lemma}
\begin{proof}
Given $f\in C^{0,\alpha}_{L,\delta}(\R^{2})$, we have to find $V\in C^{4,\alpha}_{L,\delta}(\R^{2})$ fulfilling
\begin{eqnarray}\notag
\begin{cases}
(-\Delta+\Gamma_{\varepsilon,\phi})^{2}V=f\\\notag
||V||_{C^{4,\alpha}_\delta(\R^{3})}\leq c||f||_{C^{0,\alpha}_\delta(\R^{2})}.
\end{cases}
\end{eqnarray}
In order to do so, we use Lemma \ref{prop_linpb_far} twice to find $u\in C^{2,\alpha}_{L,\delta}(\R^{2})$ and $V\in C^{2,\alpha}_{L,\delta}(\R^{2})$, such that
\begin{eqnarray}
\begin{cases}\notag
(-\Delta+\Gamma_{\varepsilon,\phi})u=f\\
(-\Delta+\Gamma_{\varepsilon,\phi})V=u,
\end{cases}
\end{eqnarray}
and 
\begin{eqnarray}\notag
\begin{cases}
||u||_{C^{2,\alpha}_\delta(\R^{2})}\leq c\,||f||_{C^{0,\alpha}_\delta(\R^{2})}\\\notag
||V||_{C^{2,\alpha}_\delta(\R^{2})}\leq c\,||u||_{C^{0,\alpha}_\delta(\R^{2})}.
\end{cases}
\end{eqnarray}
Now it remains to estimate the higher-order derivatives of $u$. For this purpose, we differentiate the equation satisfied by $u$ to get 
\begin{eqnarray}\notag
(-\Delta+\Gamma_{\varepsilon,\phi})V_{j}=u_{j}-(\Gamma_{\varepsilon,\phi})_{j}V
\end{eqnarray}
for $j=1,\dots,3$. By Proposition \ref{prop_linpb_far}, we can find a unique solution  $\tilde{V}\in C^{2,\alpha}_{L,\delta}(\R^{2})$ such that
\begin{eqnarray}\notag
||\tilde{V}||_{C^{2,\alpha}_\delta(\R^{2})}\leq c\,(||u_{j}||_{C^{0,\alpha}_\delta(\R^{2})}+||f||_{C^{0,\alpha}_\delta(\R^{2})})\leq c\,||f||_{C^{0,\alpha}_\delta(\R^{2})},
\end{eqnarray}
hence $\tilde{V}=V_j$.

Similarly, differentiating the equation once again, we see that 
\begin{eqnarray}\notag
(-\Delta+\Gamma_{\varepsilon,\phi})V_{ij}=u_{ij}-(\Gamma_{\varepsilon,\phi})_{i}V_{j}-(\Gamma_{\varepsilon,\phi})_{j}V_{i}
-(\Gamma_{\varepsilon,\phi})_{ij}V,
\end{eqnarray}
for $i,j=1,\dots,3$, so in particular $V_{ij}\in C^{4,\alpha}_{L,\delta}(\R^{3})$ and
\begin{eqnarray}\notag
||V_{ij}||_{C^{2,\alpha}_\delta(\R^{2})}\leq c\,(||u_{ij}||_{C^{0,\alpha}_\delta(\R^{2})}+||f||_{C^{0,\alpha}_\delta(\R^{2})})\leq c\,||f||_{C^{0,\alpha}_\delta(\R^{2})}.
\end{eqnarray}
This concludes the proof.
\end{proof}

\subsubsection{A fixed point argument}

Equation (\ref{eq_aux1}), whose resolvability is the purpose of this subsection, is equivalent to the fixed point problem
\begin{eqnarray}\notag
V=\mathcal{S}_{1}(V):=-\Psi_{\varepsilon,\phi}\bigg\{(1-\chi_{2})F(v_{\varepsilon,\phi})+(1-\chi_{1})Q_{\varepsilon,\phi}(\chi_{2}U+V)
+\text{N}_{\varepsilon,\phi}(U)+\text{P}_{\varepsilon,\phi}(V)\bigg\}.
\end{eqnarray}
We will solve it by showing that $\mathcal{S}_{1}$ is a contraction on the ball
\begin{eqnarray}\notag
\Lambda_{1}:=\{V\in C^{4,\alpha}_{L,\delta}(\R^{2}):||V||_{C^{4,\alpha}_\delta(\R^{2})}\leq C_{1}e^{-\delta/8\varepsilon}\},
\end{eqnarray}
provided the constant $C_{1}$ is large enough. This step of the proof is  similar to that in Section $6,2$ of \cite{Ri}. In order to prove existence, we have to show that $\mathcal{S}_1$ maps the ball into itself, provided the constant is large enough, and that it is Lipschitz continuous in $V$ with Lipschitz constant of order $\varepsilon$. The Lipschitz dependence on the data is proved exploiting the Lipschitz character of N$_{\eps,\phi}$, $Q_{\eps,\phi}$ and $\Gamma_{\eps,\phi}$ (see \eqref{quadratic_w}, \eqref{defN} and Remark \ref{rem_potential}) with respect to $U$ and $\phi$. More precisely, we use the fact that
\begin{eqnarray}\notag
||\text{N}_{\varepsilon,\phi}(U_{1})-\text{N}_{\varepsilon,\phi}(U_{2})||_{C^{0,\alpha}_\delta(\R^{2})}\leq c\,e^{-\delta/8\varepsilon}||U_{2}-U_{1}||_{D^{4,\alpha}_{\delta}(\R^2)},\\\notag
||\text{P}_{\varepsilon,\phi}(V_{\varepsilon,\phi,U_{1}})-\text{P}_{\varepsilon,\phi}(V_{\varepsilon,\phi,U_{2}})||_{C^{0,\alpha}_\delta(\R^{2})}\leq c\,e^{-\delta/8\varepsilon}||V_{\varepsilon,\phi,U_{1}}-V_{\varepsilon,\phi,U_{2}}||_{C^{4,\alpha}_\delta(\R^{2})},\\\notag
||\Gamma_{\eps,\phi_1}-\Gamma_{\eps,\phi_2}||_{C^{4,\alpha}_\delta(\R^2)}\leq c\,e^{-\delta/8\eps}||\phi_1-\phi_2||_{C^{4,\alpha}(\R)}.
\end{eqnarray}

\subsection{Proof of Proposition \ref{propaux_2}}

The aim of this section is to solve equation (\ref{proj_prob}).  Recall that we defined (see \eqref{eq:mcal-L})
\begin{eqnarray}\notag
\mathcal{L}:=-(\partial^2_s+\partial^2_t)+W^{''}(v_0(t)).
\end{eqnarray}
We first consider the linear problem
\begin{eqnarray}
\begin{cases}
\mathcal{L}^{2} \, U=f &\text{in }\R^2,\\
\int_{\R}U(s,t)v_{0}^{'}(t)dt=0 &\forall s\in\R,
\end{cases}
\label{eq_lin_near}
\end{eqnarray}
in order to produce a right inverse of $\mathcal{L}^2$ (see Subsection $5.1$). Then we apply this right inverse to define a contraction on a suitable small ball that will give us the solution through a fixed point argument.\\ 

We recall that we endowed the spaces $D^{n,\alpha}_{T,\delta}(\R^2)$ with the weighted norms
\begin{eqnarray}\notag
||U||_{D^{n,\alpha}_{\delta}(\R^2)}:=||U\psi_\delta||_{C^{n,\alpha}(\R^2)},
\end{eqnarray}
where $\psi_\delta$ is defined in (\ref{defpsi}).

\subsubsection{The linear problem}
As in Section \ref{s:red}, we first consider the second order problem
\begin{eqnarray}
\begin{cases}
\mathcal{L} \, U=f &\text{in }\R^2,\\
\int_{\R}U(y,t)v_{0}^{'}(t)dt=0 &\forall s\in\R.
\end{cases}
\label{eq_lin_near_2}
\end{eqnarray}
In order to get an estimate in a suitable weighted norm, we need an a priori estimate, that we will state in the next Lemma. This result is similar to Lemma $6.2$ in \cite{dPKW}, but here the situation is simpler since we just have exponential weights on the $t$-variable, while in \cite{dPKW} there is also a weight in the limit manifold, which is non-periodic. 

\begin{lemma}[A priori estimate]

Let $0<\delta<\sqrt{2}$, $f\in D^{0,\alpha}_{T,\delta}(\R^{2})$ and $U\in D^{2,\alpha}_{T,\delta}(\R^{2})$ be a solution to
\begin{eqnarray}\notag
\mathcal{L} \, U=f\\\notag
\int_{\R}U(s,t)v_{0}^{'}(t)dt=0, &\forall s\in\R,
\end{eqnarray}
satisfying ($L = L_T$, the $x_1$-period of $\gamma_{T}$)
\begin{eqnarray}\notag
U(x_1,x_2)=-U(-x_1,x_2)=-U\bigg(x_1+\frac{L}{2},-x_2\bigg), &\forall x=(x_1,x_2)\in\R^2,
\end{eqnarray}
and such that $U\psi_{\delta}\in L^{\infty}(\R^{2})$. Then $U\in D^{2,\alpha}_{T,\delta}(\R^{2})$ and
\begin{eqnarray}
||U||_{D^{2,\alpha}_{\delta}(\R^{2})}\leq c\,||f||_{D^{0,\alpha}_{\delta}(\R^{2})} 
\label{est_w_norm}
\end{eqnarray}
for some constant $c>0$ independent of $\varepsilon$.
\end{lemma}

\begin{proof}
As above, we set $\tilde{U}:=U\psi_{\delta}$ and $\tilde{f}:=f\psi_{\delta}$ (we recall that $\psi_\delta$ is a function of the $t$-variable). Since $\tilde{U}$ fulfils the equation
\begin{eqnarray}\notag
-\Delta\tilde{U}+2\psi_{-\delta}\partial_{t}\psi_{\delta}\partial_{t}\tilde{U}+
(W^{''}(v_{0}(t))-2(\psi_{-\delta}\partial_{t}\psi_{\delta})^{2}+\psi_{-\delta}\partial_{tt}\psi_{\delta})\tilde{U}=
\tilde{f},
\end{eqnarray}
where $\psi_{-\delta}:=\psi_\delta^{-1}$, then by elliptic estimates it is enough to show that
\begin{eqnarray}\notag
||U\psi_{\delta}||_{L^{\infty}(\R^{2})}\leq c \, ||f\psi_{\delta}||_{L^{\infty}(\R^{2})}.
\end{eqnarray}
We argue by contradiction, that is we suppose that there exists a sequence $\varepsilon_{n}\to 0$, $T_n:=\eps_n^{-1}$, $f_{n}\in D^{0,\alpha}_{T_{n},\delta}(\R^{2})$ such that $||f_{n}||_{D^{0,\alpha}_{\delta}(\R^{2})}\to 0$ and a sequence $U_{n}\in D^{2,\alpha}_{T_{n}}(\R^{2})$ of solutions to
\begin{eqnarray}\notag
\mathcal{L}\,U_{n}=f_{n} &\text{in $\R^{2};$}\\
\int_{\R}U_{n}(s,t)v_{0}^{'}(t)dt=0 &\forall s\in\R,\label{ortU_n}
\end{eqnarray}
such that $||U_{n}\psi_{\delta}||_{L^{\infty}(\R^{2})}=1$. In particular, there exists  $(s_{n},t_{n})\in\R^{2}$ such that $|U_{n}(s_{n},t_{n})|\psi_{\delta}(t_{n})\to 1$. We distinguish among three cases.\\

\textit{(i)} First we assume that $|s_{n}|+|t_{n}|$ is bounded. By the uniform bound on the norms, up to a subsequence, $U_{n}$ converges in the $C^{2}_{loc}(\R^{2})$ sense to a bounded $C^{2}(\R^{2})$-solution $U_{\infty}$ to
\begin{eqnarray}
-\Delta U_{\infty}+W^{''}(v_{0}(t))U_{\infty}=0 &\text{in $\R^{2}$}.
\label{eq_lim}
\end{eqnarray}
Hence, by Lemma $6,1$ in \cite{dPKW}, $U_{\infty}= \lambda \, v_{0}^{'}(t)$ for some $\lambda \in \R$. Moreover, by (\ref{ortU_n}),
\begin{eqnarray}
0=\int_{\R}U_{n}(s,t)v_{0}^{'}(t)dt\to\int_{\R}U_{\infty}(s,t)v_{0}^{'}(t)dt &\forall s\in\R,
\label{ort_lim}
\end{eqnarray}
thus $U_{\infty}\equiv 0$. However, up to a subsequence, $s_{n}\to s_{\infty}$ and $t_{n}\to t_{\infty}$, hence $|U_{n}(s_{n},t_{n})|\psi_{\delta}(t_{n})\to |U_{\infty}(s_{\infty},t_{\infty})|\psi_{\delta}(t_{\infty})=1$, a contradiction.\\
 
\textit{(ii)} Now we assume that $t_{n}$ is unbounded. We set
\begin{eqnarray}\notag
\tilde{U}_{n}(s,t):=U_{n}(s+s_{n},t+t_{n})\psi_{\delta}(t+t_{n}).
\end{eqnarray}
As above, exploiting the equation satisfied by $\tilde{U}_{n}$, the uniform $L^{\infty}$ bound of $U_{n}\psi_{\delta}$ and  elliptic estimates,  up to a subsequence, $\tilde{U}_{n}$ converges in  $C^{2}_{loc}(\R^{2})$  to a bounded solution $\tilde{U}_{\infty}$ to
\begin{eqnarray}\notag
-\Delta\tilde{U}_{\infty}+2\delta\partial_{t}\tilde{U}_{\infty}+(W^{''}(1)-\delta^{2})\tilde{U}_{\infty}=0.
\end{eqnarray}
By construction,
\begin{eqnarray}\notag
|\tilde{U}_{n}(0,0)|=|U_{n}(s_{n},t_{n})|\psi_{\delta}(t_{n})\to 1,
\end{eqnarray}
and $|\tilde{U}_{n}(s,t)|\leq 1$ for any $(s,t)\in\R^{2}$, thus $|\tilde{U}_{\infty}(0,0)|=1=\sup_{\R^{2}}|\tilde{U}_{\infty}|$. If, for instance, $\tilde{U}_{\infty}(0,0)=1$, then it is a maximum, hence
\begin{eqnarray}\notag
(W^{''}(1)-\delta^{2})=(W^{''}(1)-\delta^{2})\tilde{U}_{\infty}(0,0)\\\notag
\leq-\Delta\tilde{U}_{\infty}+2\delta\partial_{t}\tilde{U}_{\infty}+(W^{''}(1)-\delta^{2})\tilde{U}_{\infty}=0,
\end{eqnarray}
a contradiction. With a similar argument, we can also exclude the case $\tilde{U}_{\infty}(0,0)=-1$.\\

\textit{(iii)} It remains to rule out the case where $t_{n}$ is bounded and $s_{n}$ is unbounded. We define
\begin{eqnarray}\notag
\tilde{U}_{n}(s,t):=U_{n}(s+s_{n},t).
\end{eqnarray}
As above, we have convergence, up to a subsequence to a bounded $C^{2}$ solution to (\ref{eq_lim}). Since (\ref{ort_lim}) is still true, once again we conclude that $\tilde{U}_{\infty}\equiv 0$. Nevertheless, extracting a subsequence $t_{n}\to t_{\infty}$ if necessary, we have
\begin{eqnarray}\notag
|\tilde{U}_{n}(0,t_{n})|=|U_{n}(s_{n},t_{n})|\to\psi_{-\delta}(t_{\infty}),
\end{eqnarray}
hence $|\tilde{U}_{\infty}(0,t_{\infty})|=\psi_{-\delta}(t_{\infty})>0$, a contradiction.
\end{proof}

\begin{lemma}
Let $0<\delta<\sqrt{2}$, and let $f\in D^{0,\alpha}_{T,\delta}(\R^2)$ satisfy 
\begin{eqnarray}
\int_\R f(s,t)v_0^{'}(t)dt=0 &\forall s\in\R.
\label{ort_f}
\end{eqnarray}
Then, for $\varepsilon>0$ small enough, there exist a unique solution $U=\tilde{G}_{\varepsilon}(f)\in D^{2,\alpha}_{T,\delta}(\R^2)$ to (\ref{eq_lin_near_2}) such that 
\begin{eqnarray}\notag
||U||_{D^{2,\alpha}_{\delta}(\R^2)}\leq C||f||_{D^{0,\alpha}_{\delta}(\R^2)},
\end{eqnarray}
for some constant $C>0$ independent of $\varepsilon$ and $\phi$.
\label{prop_lin_near_2}
\end{lemma}
\begin{proof}
Exploiting the periodicity, first we look for a weak solution $Z\in H^{1}(S)$ to the problem
\begin{eqnarray}\notag
\begin{cases}
-\Delta Z+W^{''}(v_{0}(t))Z=f &\text{in $S$}; \\
\partial_{s}Z(-T,t)=\partial_{s}Z(T,t)=0 &\forall t\in\R; \\
\int_{\R}Z(s,t)v_{0}^{'}(t)dt=0 &\forall s\in(-T,T),
\end{cases}
\end{eqnarray}
where $S:=(-T,T)\times\R$, then we extend it to the whole $\R^2$. In other words, we look for a function $Z\in H^{1}(S)$ satisfying
\begin{eqnarray}\notag
\int_{S} \langle \nabla Z,\nabla v \rangle +\int_{S}W^{''}(v_{0}(t))Zv=\int_{S}fv & \forall v\in H^{1}(S),\\\notag
\int_{\R}Z(s,t)v_{0}^{'}(t)dt=0 &\text{a.e. $s\in(-T,T)$}.
\end{eqnarray}
Since
\begin{eqnarray}\notag
\int_{\R}(v^{'})^{2}+W^{''}(v_{0})v^{2}dt\geq c||v||^{2}_{H^{1}(\R)}, 
\end{eqnarray}
for any $v\in H^{1}(\R)$ such that
\begin{eqnarray}\notag
\int_{\R}vv_{0}^{'}dt=0,
\end{eqnarray}
the symmetric bilinear form defined by
\begin{eqnarray}\notag
b(Z,v):=\int_{S} \langle \nabla Z,\nabla v \rangle +\int_{S}W^{''}(v_{0}(t))Zv
\end{eqnarray}
is coercive on the closed subspace
\begin{eqnarray}\notag
X:=\bigg\{Z\in H^{1}(S):\int_{\R}Z(s,t)v_{0}^{'}(t)dt=0 &\text{a.e. $s\in(-T,T)$}\bigg\}.
\end{eqnarray}
Therefore, by the Lax-Milgram theorem, there exists a unique $Z\in X$ such that 
\begin{eqnarray}
b(Z,v)=\int_{S}fv &\forall v\in X.
\label{w_sol_X}
\end{eqnarray}
In order to show that $Z$ is actually a weak solution, we need to prove that (\ref{w_sol_X}) is true for any $v\in H^{1}(S)$. In order to do so, we decompose an arbitrary $v\in H^{1}(S)$ as
\begin{eqnarray}\notag
v(s,t)=\tilde{v}(s,t)+c(s)v_{0}^{'}(t),
\end{eqnarray}
where $c(s):=\int_{\R}vv_{0}^{'}dt/\int_{\R}(v_{0}^{'})^{2}dt$ is chosen in such a way that $\tilde{v}\in X$. We observe that, since 
\begin{eqnarray}\notag
\int_{S}f(s,t)v_{0}^{'}(t)dt=0, &\forall s\in(-T,T),
\end{eqnarray}
we have
\begin{eqnarray}\notag
\int_{-T}^{T}c(s)ds\int_{\R}f(s,t)v_{0}^{'}(t)dt=0.
\end{eqnarray}
Moreover, an integration by parts and Fubini-Tonelli's Theorem show that
\begin{eqnarray}\notag
b(Z,cv_{0}^{'})=\int_{-T}^{T}\partial_{s}c(s)\partial_{s}\int_{\R}Z(s,t)v_{0}^{'}(t)dt+
\int_{-T}^{T}c(s)\int_{\R}ZL_{\star}v_{0}^{'}dt=0.
\end{eqnarray}
In conclusion,
\begin{eqnarray}\notag
b(Z,v)=b(Z,\tilde{v})+b(Z,cv_{0}^{'})=\int_{S}f\tilde{v}dt=\int_{S}fv.
\end{eqnarray}
In order to prove symmetry and to extend $Z$ to an entire solution $U\in C^{2,\alpha}(\R^{2})$, see \textit{Step (ii)} of the proof of Lemma \ref{prop_linpb_far}. Arguing as in \textit{Step (iii)} of that proof, it is possible to show that $U\in L^{\infty}(\R^{2})$. In order to show that $U\psi_{\delta}\in L^{\infty}(\R^{2})$, we use the function $\lambda e^{-\delta |t|}+\sigma e^{\delta |t|}$ as a barrier, 
for suitable constant $\lambda$ and $\tau$. Here there is a slight difference with respect to the proof of Lemma \ref{prop_linpb_far}, due the fact that the potential is not uniformly positive. This is actually not so relevant, since $W^{''}(v_{0}(t))$ is close to $W^{''}(1) = 2$ for $|t|$ large enough.
\end{proof}
Now we can solve the fourth-order problem (\ref{eq_lin_near}), by applying iteratively Lemma \ref{prop_lin_near_2}.
\begin{lemma}
Let $0<\delta<\sqrt{2}$ and let $f\in D^{0,\alpha}_{T,\delta}(\R^2)$ satisfy (\ref{ort_f}). Then there exists a unique solution $U=G_{\varepsilon}(f)\in D^{4,\alpha}_{T,\delta}(\R^2)$ to (\ref{eq_lin_near}) such that 
\begin{eqnarray}\notag
||U||_{D^{4,\alpha}_{\delta}(\R^2)}\leq C||f||_{D^{0,\alpha}_{\delta}(\R^2)},
\end{eqnarray}
for some constant $C>0$ independent of $\varepsilon$.
\end{lemma}

\subsubsection{Proof of Proposition \ref{propaux_2} completed}

The proof is based on a fixed point argument. In fact, we have to find a fixed point of the map
\begin{equation}\label{eq:??}
\mathcal{S}_{2}(U):=G_{\varepsilon}\bigg\{-\chi_4 F(\tilde{v}_{\varepsilon,\phi})-\text{T}(U,V_{\varepsilon,\phi,U},\phi)+p_\phi(y)v_{0}^{'}(t)\bigg\}
\end{equation}
on a suitable small metric ball of the form
\begin{eqnarray}\notag
\Lambda_{2}:=\bigg\{U\in D^{4,\alpha}_{\delta}(\R^2):\int_\R U(s,t)v_0^{'}(t)dt=0 &\forall s\in\R,||U||_{D^{4,\alpha}_{\delta}(\R^2)}\leq C_{2}\varepsilon^5\bigg\},
\end{eqnarray}
provided $C_{2}>0$ is large enough. Once again, we will prove that $\mathcal{S}_2$ is a contraction $\Lambda_2$. First we observe that, by definition of $p_\phi$, the quantity inside brackets in \eqref{eq:??} is orthogonal to $v_{0}^{'}(t)$ for any $s\in\R$, thus we can actually apply the operator $G_{\varepsilon}$. Moreover, if $U$ respects the symmetries of the curve, then also the right-hand side does, hence $\mathcal{S}_{2}(U)$ respects the symmetries.

In order to prove that $\mathcal{S}_2$ is a contraction, we note that
\begin{eqnarray}\notag
||F(\tilde{v}_{\varepsilon,\phi})||_{D^{0,\alpha}_{\delta}(\R^2)}\leq \tilde{c}\,\varepsilon^5,
\end{eqnarray}
(see \eqref{error_eps5}), and a similar estimate is true for $p_\phi(s)v_{0}^{'}(t)$. The term $\text{T}(U,V_{\eps,\phi,U},\phi)$ defined in (\ref{defT}) is smaller. For instance, using (\ref{quadratic_w}) and the fact that $V$ is exponentially small, one has that 
\begin{eqnarray}\notag
||\chi_{1}Q_{\varepsilon,\phi}(U+V)||_{D^{0,\alpha}_{\delta}(\R^2)}\leq c\,\varepsilon^{10}.
\end{eqnarray}
Similarly, we can see that $||\text{M}_{\varepsilon,\phi}(V)||_{D^{0,\alpha}_{\delta}(\R^2)}\leq c\,e^{-\delta/4\varepsilon}$. In addition, since all the coefficients of $\text{R}_{\varepsilon,\phi}$ are at least of order $\varepsilon$, we get that
\begin{eqnarray}\notag
||\text{R}_{\varepsilon,\phi}(U)||_{D^{0,\alpha}_{\delta}(\R^2)}\leq c\,\eps||U||_{D^{4,\alpha}_{\delta}(\R^2)}\leq c\,\varepsilon^6.
\end{eqnarray}
For the definitions of $\text{M}_{\eps,\phi},R_{\eps,\phi}$ and $Q_{\eps,\phi}$, see (\ref{defM}), (\ref{defR}) and (\ref{quadratic_w}).

As regards the Lipschitz dependence on $U$, we observe that 
\begin{eqnarray}\notag
||\chi_{1}(Q_{\varepsilon,\phi}(U_{1}+V)-Q_{\varepsilon,\phi}(U_{2}+V))||_{D^{0,\alpha}_{\delta}(\R^2)}\leq c\,\varepsilon^{5}||U_{1}-U_{2}||_{D^{4,\alpha}_{\delta}(\R^2)}
\end{eqnarray}
and
\begin{eqnarray}\notag
||\text{R}_{\varepsilon,\phi}(U_{1})-\text{R}_{\varepsilon,\phi}(U_{2})||_{D^{0,\alpha}_{\delta}(\R^2)}\leq c\,\varepsilon||U_{1}-U_{2}||_{D^{4,\alpha}_{\delta}(\R^2)}.
\end{eqnarray}
It follows from the Lipschitz character of the potential $W$ that the solution $U$ depends on $\phi$ in a Lipschitz way.


\section{Appendix}\label{s:app}
In this section we provide a full proof of our claims from
Proposition~\ref{p:inv-rhs}, and in particular of \eqref{eq:c0c1}. To this end, we consider the solutions $\Phi(z) = \sum_{m=0}^\infty \mu_m z^{2m+1}$
of the ODE \eqref{eq:ODE_for_Phi}. A coefficient comparison yields that  $\mu_2,\mu_3,\ldots$ are
explicitly given in terms of $\mu_0,\mu_1$ via the formulas 
\begin{align} \label{eq:defn_cm}
  \begin{aligned}
  \mu_{2m} &= - \frac{3\pi \sqrt{2}}{32\Gamma(\frac{3}{4})^2} \cdot \frac{ 
      \Gamma(m-\frac{1}{4})}{4^m\Gamma(m+\frac{5}{4})} -  \frac{\mu_0}{\sqrt{\pi}}  \cdot   
  	  \frac{\Gamma(m+\frac{1}{2})}{4^m\Gamma(m+1)(4m-1)},  \\
      \mu_{2m+1} &= \frac{3\sqrt{2}\Gamma(\frac{3}{4})^2 }{8\pi} \cdot \frac{ 
  	  \Gamma(m+\frac{1}{4})}{4^m\Gamma(m+\frac{7}{4})} \, \mu_1.
  \end{aligned}
\end{align}
By the asymptotics of the Gamma function, see e.g. \cite{ET}, the convergence radius of this series is $\sqrt{2}$.   Reasoning as in Proposition~\ref{p:inv-rhs}
let us now derive two  equations for $\mu_0,\mu_1$ so that any corresponding solution
$\bar\phi=\Phi\circ k$ satisfies $\bar\phi'(s),\bar\phi'''(s)\to 0$ as $|k(s)|\to \sqrt{2}$, i.e. as $|s|\to
{\bar{T}}/4$. It will turn out that the solution of this system is unique and given by
$\mu_0=0,\mu_1= \frac{\pi^2}{8\Gamma(\frac{3}{4})^4}$,  in accordance with \eqref{eq:c0c1}. 

\

\begin{pfn} {\em of \eqref{eq:c0c1}}
To this end we first calculate the derivatives of $\bar{\phi}$. For all $m\in\N_0$ we set
\begin{align*}
  a_m &:= (2m+1)\mu_m, \\ 
  b_m &:= -\frac{1}{2}(2m+3)(2m+1)(m+1)\mu_m + 2(2m+5)(m+2)(2m+3)\mu_{m+2}.  
\end{align*} 
Then, for all $s\in (-{\bar{T}}/4,{\bar{T}}/4)$ we have $|k(s)|<\sqrt{2}$ and thus we obtain from \eqref{eq:Will} and
\eqref{eq:cons}
\begin{align*}
  \bar\phi'(s)
  &= \sum_{m=0}^\infty a_m k'(s)k(s)^{2m}, \\
  \bar\phi''(s)
  &= \sum_{m=1}^\infty a_m \big( k''(s)k(s)^{2m}+2m k'(s)^2k(s)^{2m-1} \big) + a_0k''(s)\\ 
  &= \sum_{m=1}^\infty a_m \Big(
  -\frac{1}{2}k(s)^{2m+3}+2mk(s)^{2m-1}(1-\frac{1}{4}k(s)^4)\Big) - \frac{a_0}{2}k(s)^3 \\
  &=  \sum_{m=0}^\infty  \Big(- \frac{1}{2}(m+1)a_m + 2(m+2)a_{m+2} \Big) k(s)^{2m+3} 
   + 2a_1k(s)  \\
  \bar\phi'''(s)
  &= \sum_{m=0}^\infty b_m k'(s)k(s)^{2m+2} +  2a_1k'(s). 
\end{align*}
 
\medskip
 
For the  analysis of convergence,  we rewrite $\bar\phi',\bar\phi'''$ as follows:
\begin{align*}
  \bar\phi'(s)
  &= k'(s)\cdot \sum_{m=0}^\infty \Big( a_{2m}+2a_{2m+1} + a_{2m+1}(k(s)^2-2)\Big)k(s)^{4m}, \\
  \bar\phi'''(s)
  &= k'(s)k(s)^2\cdot \sum_{m=0}^\infty \Big( b_{2m}+2b_{2m+1} + b_{2m+1}(k(s)^2-2)\Big)k(s)^{4m}  + 
  2a_1k'(s).
\end{align*}
Therefore we have to investigate the behaviour of the terms
$a_{2m}+2a_{2m+1},b_{2m}+2b_{2m+1},a_{2m+1},b_{2m+1}$ as $m\to\infty$. To this end we use the known
asymptotics (see p.1 in \cite{ET})
\begin{equation} \label{eq:Gamma_asymptotics}
  \frac{\Gamma(z+\alpha)}{\Gamma(z+\beta)} 
  = z^{\alpha-\beta}\cdot \Big( 1 + \frac{(\alpha-\beta)(\alpha+\beta-1)}{2z} + O(z^{-2})\Big)
  \quad\text{as }z\to\infty   
\end{equation}
for any fixed $\alpha,\beta\in\R$. 

\medskip

We start with analysing the behaviour of $\bar\phi'(s)$ as $|s|\to {\bar{T}}/4$. We have  
\begin{align*}
  a_{2m}+2a_{2m+1}
  &= (4m+1)\mu_{2m} + 2(4m+3)\mu_{2m+1} \\
  &= (4m+1)\cdot \Big(  
  - \frac{3\pi \sqrt{2}}{32\Gamma(\frac{3}{4})^2} \cdot \frac{ 
      \Gamma(m-\frac{1}{4})}{4^m\Gamma(m+\frac{5}{4})} -  \frac{\mu_0}{\sqrt{\pi}}  \cdot   
  	  \frac{\Gamma(m+\frac{1}{2})}{4^m\Gamma(m+1)(4m-1)}\Big) \\
  &\quad + 2(4m+3)\cdot \frac{3\sqrt{2}\Gamma(\frac{3}{4})^2 \mu_1}{8\pi} \cdot \frac{ 
  	  \Gamma(m+\frac{1}{4})}{4^m\Gamma(m+\frac{7}{4})} \\
  &= \frac{1}{4^m} \cdot \Big( - \frac{3\pi \sqrt{2}}{8\Gamma(\frac{3}{4})^2} \cdot \frac{ 
      \Gamma(m-\frac{1}{4})}{\Gamma(m+\frac{1}{4})} -  \frac{\mu_0}{\sqrt{\pi}}  \cdot   
  	  \frac{\Gamma(m+\frac{1}{2})(4m+1)}{\Gamma(m+1)(4m-1)} \\
  &\quad +  \frac{3\sqrt{2}\Gamma(\frac{3}{4})^2 \mu_1}{\pi} \cdot \frac{ 
  	  \Gamma(m+\frac{1}{4})}{ \Gamma(m+\frac{3}{4})}\Big) \\
  &= \frac{1}{4^m m^{1/2}} \cdot  \Big(- \frac{3\pi \sqrt{2}}{8\Gamma(\frac{3}{4})^2} - 
  \frac{1}{\sqrt{\pi}}\cdot \mu_0   + \frac{3\sqrt{2}\Gamma(\frac{3}{4})^2 }{\pi}\cdot \mu_1 + O(m^{-1})\Big), \\
  a_{2m+1}
  &= \frac{3\sqrt{2}\Gamma(\frac{3}{4})^2\mu_1}{2\pi}\cdot
  \frac{\Gamma(m+\frac{1}{4})}{4^m\Gamma(m+\frac{3}{4})} = \frac{1}{4^m}\cdot O(m^{-1/2}). 
\end{align*}
Therefore we must require
\begin{equation} \label{eq:req1}
  - \frac{3\pi \sqrt{2}}{8\Gamma(\frac{3}{4})^2} -  \frac{1}{\sqrt{\pi}}\cdot \mu_0   + \frac{3\sqrt{2}\Gamma(\frac{3}{4})^2 }{\pi}\cdot \mu_1
  = 0 
\end{equation}
Once this equation is satisfied we have for some positive $C,C'$ as $|s|\to {\bar{T}}/4$
\begin{align*}
  |\bar\phi'(s)|
  &\leq |k'(s)|\cdot \sum_{m=0}^\infty \Big( 4^m|a_{2m}+2a_{2m+1}| + 4^m|a_{2m+1}||k(s)^2-2|\Big)
  \Big(\frac{k(s)^4}{4}\Big)^m  \\
  &\leq  c|k'(s)| \cdot \Big( \sum_{m=0}^\infty m^{-3/2} + 
  \sum_{m=0}^\infty |k(s)^2-2|\Big(\frac{k(s)^4}{4}\Big)^m \Big) \\
  &\leq  c'|k'(s)| \cdot \Big( 1 + |k(s)^2-2|\cdot \frac{1}{1-\frac{k(s)^4}{4}}  \Big) 
  = o(1), 
\end{align*} 
so  the desired asymptotics for $\bar\phi'$ holds. Hence, we have shown that any couple $\mu_0,\mu_1$ satisfying
\eqref{eq:req1} yields the convergence of   $\bar\phi'(s)$ as $s \to \bar{T}/4$.

\medskip

Now we turn to the third-order derivatives. We have
\begin{align*}
  b_{2m}
  &= - \frac{1}{2}(4m+3)(4m+1)(2m+1)\mu_{2m} + 2(4m+5)(2m+2)(4m+3)\mu_{2m+2} \\
  &=  \frac{3\pi\sqrt{2}}{64\Gamma(\frac{3}{4})^2}\cdot 
    \frac{(4m+3)(4m+1)(2m+1)\Gamma(m-\frac{1}{4})}{4^m\Gamma(m+\frac{5}{4})}  \\
  &\quad - \frac{3\pi\sqrt{2}}{16\Gamma(\frac{3}{4})^2} \cdot
      \frac{(4m+5)(2m+2)(4m+3)\Gamma(m+\frac{3}{4})}{4^{m+1}\Gamma(m+\frac{9}{4})}  \\
  &\quad + \frac{\mu_0}{2\sqrt\pi}\cdot  
  \frac{(4m+3)(4m+1)(2m+1)\Gamma(m+\frac{1}{2})}{4^m\Gamma(m+1)(4m-1)} \\ 
  &\quad -  \frac{2\mu_0}{ \sqrt\pi}\cdot 
  \frac{ (4m+5)(2m+2)(4m+3)\Gamma(m+\frac{3}{2})}{4^{m+1}\Gamma(m+2)(4m+3)}
    \\
  &=    \frac{3\pi\sqrt{2}}{16\Gamma(\frac{3}{4})^2}\cdot \frac{1}{4^m} \Big(
    \frac{(4m+3)(2m+1)\Gamma(m-\frac{1}{4})}{\Gamma(m+\frac{1}{4})}  
    -  \frac{(4m+3)(2m+2)\Gamma(m+\frac{3}{4})}{\Gamma(m+\frac{5}{4})}    
    \Big) \\
  &\quad + \frac{\mu_0}{2\sqrt\pi}\cdot \frac{1}{4^m}  
  \Big( 
    \frac{(4m+3)(4m+1)(2m+1)\Gamma(m+\frac{1}{2})}{(4m-1)\Gamma(m+1)}  
    -  \frac{(4m+5)(2m+2)\Gamma(m+\frac{3}{2})}{\Gamma(m+2)}    
    \Big), 
\end{align*}
as well as 
\begin{align*}
  b_{2m+1}
  &=   - \frac{1}{2}(4m+5)(4m+3)(2m+2) \mu_{2m+1} + 2(4m+7)(2m+3)(4m+5)\mu_{2m+3} \\
  &=  -\frac{3\sqrt{2}\Gamma(\frac{3}{4})^2\mu_1}{16\pi} \cdot 
  \frac{(4m+5)(4m+3)(2m+2)\Gamma(m+\frac{1}{4})}{4^m \Gamma(m+\frac{7}{4})}  \\
   &\quad + \frac{3\sqrt{2}\Gamma(\frac{3}{4})^2\mu_1}{4\pi}
   \cdot \frac{4(4m+7)(2m+3)(4m+5)\Gamma(m+\frac{5}{4})}{4^{m+1}\Gamma(m+\frac{11}{4})}  
  \\
  &=   \frac{3\sqrt{2}\Gamma(\frac{3}{4})^2\mu_1}{4\pi} \cdot \frac{1}{4^m}
  \Big( - \frac{(4m+5)(2m+2)\Gamma(m+\frac{1}{4})}{\Gamma(m+\frac{3}{4})}  
    + \frac{(2m+3)(4m+5)\Gamma(m+\frac{5}{4})}{\Gamma(m+\frac{7}{4})} \Big). 
\end{align*}
Using the asymptotics of the Gamma function, from \eqref{eq:Gamma_asymptotics} we obtain 
\begin{align*}
  b_{2m}+2b_{2m+1}
  &=  \frac{1}{4^m m^{1/2}} \cdot \Big( \frac{9\pi\sqrt{2}}{16\Gamma(\frac{3}{4})^2}   +
  \frac{2}{\sqrt{\pi}}\cdot \mu_0 - \frac{9\sqrt{2}\Gamma(\frac{3}{4})^2}{2\pi}\cdot \mu_1 + O(m^{-1})\Big), \\
  b_{2m+1}
  &=   \frac{1}{4^m}\cdot O(m^{-1/2}). 
\end{align*}  
 This leads us to require 
 \begin{equation} \label{eq:req2}
  \frac{9\pi\sqrt{2}}{16\Gamma(\frac{3}{4})^2}   +
  \frac{2}{\sqrt{\pi}}\cdot \mu_0 - \frac{9\sqrt{2}\Gamma(\frac{3}{4})^2}{2\pi}\cdot \mu_1
  = 0. 
\end{equation}
 As before we obtain that every $\mu_0,\mu_1$ satisfying \eqref{eq:req2} makes sure that 
 $\bar\phi'''(s)$ tends to zero as $s \to \bar{T}/4$, i.e. as  $|k(s)|\to \sqrt{2}$. 
 
 \medskip
 
 Collecting all the above reasoning, we find that  $\bar\phi'(s),\bar\phi'''(s)\to 0$ as $|k(s)|\to\sqrt{2}$ provided
 $\mu_0,\mu_1$ solve \eqref{eq:req1},\eqref{eq:req2}, which is a linear system with a unique solution $\mu_0=0,\mu_1 = 
 \frac{\pi^2}{8\Gamma(\frac{3}{4})^4}$. Plugging these values into the formula for $\bar\phi(s)=\Phi(k(s))$ we
 get
 \begin{align*}
      \bar\phi(s) 
      &=   \frac{3\pi\sqrt{2}}{64\Gamma(\frac{3}{4})^2}\cdot
      \sum_{m=0}^\infty \Big( -  \frac{\Gamma(m+\frac{3}{4})}{2\cdot 4^m\Gamma(m+\frac{9}{4})}  
      k(s)^{4m+5} + \frac{ \Gamma(m+\frac{1}{4})}{4^m\Gamma(m+\frac{7}{4})} k(s)^{4m+3} \Big),
	\end{align*}
which is precisely the formula from Proposition \ref{p:inv-rhs}. 
\end{pfn}


\begin{thebibliography}{10}



\bibitem{AS} M. Abramowitz, I.A. Stegun, \emph{Handbook of mathematical functions with formulas, graphs, and mathematical tables}, National Bureau of Standards Applied Mathematics Series, 55. 

\bibitem{ADW} O. Agudelo, M. del Pino, J. Wei, \emph{Solutions with multiple catenoidal ends to the Allen-Cahn equation in $\R^3$}, J. Math. Pures Appl. (9) 103 (2015), no. 1, 142-218. 

\bibitem{AAC} G. Alberti, L. Ambrosio, X. Cabre, \emph{On a long-standing conjecture of E. De Giorgi:
symmetry in 3D for general nonlinearities and a local minimality property},  Acta Appl. Math. 65 (2001), no. 1-3, 9-33.

\bibitem{AG} S. Alinhac, P. Gérard, \emph{Opérateurs pseudo-différentiels et théorème de Nash-Moser}, InterEditions/Editions du CNRS, 1991.

\bibitem{AC} S. M. Allen, J. W. Cahn, \emph{Ground State Structures in Ordered Binary Alloys with Second Neighbor Interactions}, Acta Met. 20, 423 (1972), 
no. 6,  921-1107. 

\bibitem{AA} G. E. Andrews, R. R. Askey, \emph{Special functions}, Encyclopedia of Mathematics and its Applications, Vol 71, Cambridge University Press, Cambridge, (1999).

\bibitem{BBG} M. Barlow, R. Bass, C. Gui, \emph{The Liouville property and a conjecture of De Giorgi}, Comm. Pure Appl. Math. 53 (2000), 1007-1038.

\bibitem{BK} M. Bauer, E. Kuwert, \emph{Existence of minimizing Willmore surfaces of prescribed genus}, Int. Math. Res. Not. 2003, no. 10, 553-576. 

\bibitem{BW} R. Beals, R. Wong, \emph{Special functions}, A graduate text. Cambridge Studies in Advanced Mathematics, 126. Cambridge University Press, Cambridge, 2010. 



\bibitem{BP} G. Bellettini, M. Paolini, \emph{Approssimazione variazionale di funzioni con curvatura}, Seminario di analisi matematica, Univ. Bologna, 1993.

\bibitem{BHM}  H. Beresticky, F. Hamel, R. Monneau, {\em One-dimensional symmetry of bounded entire solutions of some elliptic equations}, Duke Math. J. 103 (1999), 375-396. 

\bibitem{BDG} E. Bombieri, E. De Giorgi, E. Giusti, \emph{Minimal cones and the Bernstein problem} Invent. Math. 7 (1969), 243-268. 

\bibitem{CT}  X. Cabré, J. Terra, \emph{Saddle-shaped solutions of bistable diffusion equations in all of $\R^{2m}$}, J. Eur. Math. Soc.  11 (2009), no. 4, 819-843.


\bibitem{CH} J. W. Cahn, J. E. Hilliard, \emph{Free energy of a nonuniform system}, I. Interfacial free energy, J. Chem. Phys 28, 258 (1958), 258-267.

\bibitem{Dan} M. Danchin, \emph{Fourier analysis method for PDEs}, (2005). Available at perso-math.univ-mlv.fr/users/danchin.raphael/cours/courschine.pdf

\bibitem{DFP} H. Dang, P. Fife, L. A. Peletier, \emph{Saddle solutions of the bistable diffusion equation}, 
Z. Angew. Math. Phys. 43 (1992), no. 6, 984-998. 


\bibitem{DKPW}  M. del Pino, M. Kowalczyk, F. Pacard, J. Wei, \emph{Multiple-end solutions to the Allen-Cahn equation in $\R^2$}, J. Funct. Anal. 258 (2010), no. 2, 458-503. 

\bibitem{dPKW} M. del Pino, M. Kowalczyk, J. Wei, \emph{On De Giorgi's conjecture in dimension $N\geq 9$}, Ann. of Math. 174 (2011), 1485-1569.

\bibitem{DMP}  M. del Pino, M. Musso, F. Pacard, \emph{Solutions of the Allen-Cahn equation which are invariant under screw-motion}, Manuscripta Math. 138 (2012), no. 3-4, 273-286. 



\bibitem{DG} E. De Giorgi, \emph{Convergence problems for functionals and operators, Proc. Internat. Meeting
on Recent Methods in Nonlinear Analysis}, (Rome, 1978) (E. De Giorgi et al., eds.),
Pitagora, Bologna, 1979, 131-188.

\bibitem{ET} A. Erdélyi, F. G. Tricomi, \emph{The asymptotic expansion of a ratio of gamma functions},
 'Pacific J. Math. Volume 1,
Number 1 (1951), 133-142.
  

\bibitem{Fa} A. Farina, \emph{Symmetry for solutions of semilinear elliptic equations in $R^N$ and related
conjectures}, Ricerche Math. 48 (1999), 129-154.

\bibitem{GrG}     F. Gazzola, H.-C. Grunau, G. Sweers,  Polyharmonic boundary value problems. Positivity preserving and nonlinear higher order elliptic equations in bounded domains. Lecture Notes in Mathematics, 1991. Springer-Verlag, Berlin, 2010.

\bibitem{GG} N. Ghoussoub, C. Gui, \emph{On a conjecture of De Giorgi and some related problems},
Math. Ann. 311 (1998), 481-491.





\bibitem{GT} D. Gilbarg, N. S. Trudinger, \emph{Elliptic partial differential equations of second order,} second edition, Grundlehren der Mathematischen Wissenschaften, 224, Springer Verlag, Berlin, 1983.


\bibitem{HT} J. Hutchinson, Y. Tonegawa,
\emph{Convergence of phase interfaces in the van der Waals-Cahn-Hilliard theory}. 
Calc. Var. Partial Differential Equations 10 (2000), no. 1, 49-84. 

\bibitem{LMS} T. Lamm, J. Metzger, F. Schulze, \emph{Foliations of asymptotically flat manifolds by surfaces of Willmore type}, Math. Ann. 350 (2011), no. 1, 1-78.


\bibitem{MW} A. Malchiodi, J. Wei \emph{Boundary interface for the Allen-Cahn equation}. J. fixed
Point Theory Appl. 1, (2007), no. 2, 305-336.

\bibitem{Ma}  R. Mandel, \emph{Boundary value problems for Willmore curves in $\R^2$}, Calc. Var. Partial Differential Equations 54 (2015), no. 4, 3905-3925.

\bibitem{MN} F. Marques, A. Neves, \emph{Min-max theory and the Willmore conjecture}. Ann. of Math. (2) 179 (2014), no. 2, 683-782.


\bibitem{Mo} L. Modica, \emph{The gradient theory of phase transitions and the minimal interface criterion}.
Arch. Rational Mech. Anal. 98 (1987), no. 2, 123-142. 

\bibitem{MM} L. Modica, S. Mortola, \emph{Un esempio di $\Gamma$-convergenza},  Boll. Un. Mat. Ital. B (5) 14 (1977), no. 1, 285-299. 



\bibitem{NT} Y. Nagase, Y. Tonegawa, \emph{A singular perturbation problem with integral curvature bound}, Hiroshima Math. J. 37, (2007), no. 3, 455-489.

\bibitem{PR} F. Pacard, M. Ritoré, \emph{From constant mean curvature hypersurfaces to the gradient theory of phase transitions}, 
J. Differential Geom. 64 (2003), no. 3, 359-423. 

\bibitem{Ri} M. Rizzi, \emph{Clifford Tori and the singularly perturbed Cahn-Hilliard equation}, J. Diff. Eq., to appear. 

\bibitem{RS} M. Röger, R. Schätzle, \emph{On a modified conjecture of De Giorgi}, Math. Z. 254, (2006), no. 4, 675-714.

\bibitem{Sa} O. Savin, \emph{Regularity of flat level sets in phase transitions}, Ann. of Math. (2) 169 (2009), no. 1, 41-78. 

\bibitem{St}  P. Sternberg, 
\emph{The effect of a singular perturbation on nonconvex variational problems}.
Arch. Rational Mech. Anal. 101 (1988), no. 3, 209-260. 

\end{thebibliography}
\end{document}